\pdfoutput=1
\documentclass[11pt,leqno]{article}

\usepackage{amsmath,amsthm,amssymb}
\usepackage{authblk,subfigure,booktabs,enumerate}
\usepackage{graphicx}
\usepackage{color}
\usepackage{mathabx,bm}
\usepackage{algorithm,algpseudocode}
\usepackage{url}
\usepackage{enumitem}

\usepackage{geometry}
\algnewcommand\algorithmicinput{\textbf{Input:}}
\algnewcommand\Input{\item[\algorithmicinput]}
\algnewcommand\algorithmicoutput{\textbf{Output:}}
\algnewcommand\Output{\item[\algorithmicoutput]}

\theoremstyle{plain}
\newtheorem{thm}{Theorem}[section]
\newtheorem{prop}[thm]{Proposition}
\newtheorem{cor}[thm]{Corollary}
\newtheorem{lem}[thm]{Lemma}

\theoremstyle{remark}
\newtheorem{rem}[thm]{Remark}

\theoremstyle{definition}
\newtheorem{ex}[thm]{Example}
\newtheorem{defn}[thm]{Definition}

\makeatletter

\numberwithin{equation}{section}
\numberwithin{figure}{section}
\numberwithin{table}{section}
\numberwithin{algorithm}{section}

\newcommand{\EE}{\mathbb{E}}
\newcommand{\RR}{\mathbb{R}}
\newcommand{\PP}{\mathbb{P}}

\newcommand{\cP}{\mathcal{P}}
\newcommand{\cH}{\mathcal{H}}

\title{Kernel-based potential mean-field games \\ with unbiased random Fourier $U$-statistics}

\author[1]{Yumiharu Nakano\thanks{E-mail: nakano@comp.isct.ac.jp}}
\affil[1]{Department of Mathematical and Computing Science,
            Institute of Science Tokyo}

\date{\today}

\begin{document}

\maketitle

\begin{abstract}
We study the subclass of potential mean-field games
in which the running interaction cost and the terminal target cost are both expressed
through reproducing-kernel maximum mean discrepancy (MMD) penalties,
and develop a computational framework that exploits this kernel structure.
Both costs are estimated from finite-sample empirical distributions
using a random Fourier U-statistic representation that is unbiased
and has linear cost in the batch size.
The drift of the controlled diffusion is parametrized by a neural network
and trained via stochastic gradient descent.
For population near-minimizers we prove convergence to the terminal-constrained problem
as the penalty diverges, and show that the same limit is recovered almost surely
when such controls are evaluated on independent finite samples under explicit
coupling conditions on the penalty, random-feature count and sample size.
The framework includes the kernel-MMD-penalty Schr{\"o}dinger bridge problem
as the special case of a vanishing interaction cost.
Numerical experiments illustrate the method on the Schr{\"o}dinger bridge problem
in dimensions up to one hundred,
and on an electric vehicle charging coordination problem with per-vehicle physical heterogeneity,
where a control-dependent aggregate-demand cost produces a trade-off between
peak normalized charging power and the terminal service level.

\begin{flushleft}
\end{flushleft}
\begin{flushleft}
\end{flushleft}
\end{abstract}


\section{Introduction}\label{sec:1}

Mean-field games (MFGs)
\cite{las-lio:2007,hua-cai-mal:2006}
model large populations of strategic agents through the interaction
of a representative agent with the population distribution;
see the monographs \cite{car-del:2018a,car-del:2018b} for a comprehensive account.
The subclass of \emph{potential} (or \emph{cooperative}) MFGs,
where the running and terminal costs admit variational potentials with respect to the population law,
reduces to a single mean-field stochastic optimal control problem
and has been used to model swarm coordination, traffic flow, opinion dynamics,
and energy systems (see, e.g., \cite{cou-per-tem-deb:2012,ma-cal-his:2013,par-gen-lyg:2020}).
A key variant is the \emph{planning} (or \emph{coordination}) MFG
\cite{ach-cam-cap:2012}, which prescribes
$\mathcal{L}(X_T)=\mu_T$.  Its potential formulation is also called
\emph{mean-field optimal transport} (MFOT)
\cite{ben-car-hat:2019,car-mes-san:2017}: kinetic and running mean-field costs
are minimised over flows from $\mu_0$ to $\mu_T$, with the Schr{\"o}dinger
bridge as the vanishing-interaction entropic case \cite{nak:2024sb}.
We study its soft-target version, replacing the terminal constraint by
$\lambda\gamma_K^2(\mathcal{L}(X_T),\mu_T)$.

A motivating example from the smart-grid literature is coordinated charging of electric vehicle (EV) fleets \cite{cou-per-tem-deb:2012}.
A large number of EVs decide, over an operational horizon $[0,T]$,
when and how fast to charge.
Each EV faces a private charging effort and is exposed to a market electricity price
that depends on the aggregate demand of the fleet (the population distribution of charging power),
producing a congestion-type cost.
At the deadline $T$, the fleet should reach a desired distribution of state-of-charge.
For example, all vehicles charged to a sufficient level for the next day.
This is a soft target on the terminal population law
$\mathcal{L}(X_T)$ (we write $\mathcal{L}(X)$ for the probability distribution of a random variable $X$ throughout this paper),
rather than on individual trajectories.
In an idealized state-law formulation, this leads to the potential MFG
\begin{equation}\label{eq:1.mfg_intro}
 \inf_{u\in\mathcal{A}} 
 \int_0^T \left(\frac{1}{2}\EE\bigl[|u(t,X_t)|^2\bigr] + \mathcal{R}[\mathcal{L}(X_t)]\right)\,dt
 + \mathcal{T}[\mathcal{L}(X_T)],
\end{equation}
where $X$ is a controlled diffusion driven by the drift $u$,
$\mathcal{R}$ encodes the congestion cost,
and $\mathcal{T}$ encodes the deadline target.
Classical grid-based numerical methods for MFGs, e.g., HJB--Fokker--Planck coupled solvers
\cite{ach-cap:2010,ach-cam-cap:2012} and related schemes surveyed in \cite{lau:2021}, are limited to dimensions $d\le 5$ by the curse of dimensionality,
whereas deep-learning approaches based on neural backward SDEs, fictitious play, or alternating-network architectures
\cite{de-etal:2021,car-lau:2022,rut-etal:2020,lin-etal:2021apac}
have demonstrated higher-dimensional capability empirically
but typically lack rigorous convergence guarantees in the sample-based regime where they are actually trained.

In this paper, we develop a computational framework for the potential MFG \eqref{eq:1.mfg_intro}
where both the running interaction cost $\mathcal{R}$ and the terminal cost $\mathcal{T}$ are expressed
through reproducing-kernel maximum mean discrepancies (MMDs).
The per-iteration cost of the empirical objective is linear in the sample size $N$
via the random Fourier $U$-statistic representation introduced in Section~\ref{sec:3},
in contrast to the $O(dN^2)$ cost of the original kernel-MMD-penalty approach of \cite{nak:2024sb}.
Specifically, we take
\begin{equation}\label{eq:1.kernel_costs}
 \mathcal{R}[\nu_t] = \frac{c}{2}\iint W(x,y)\,\nu_t(dx)\,\nu_t(dy),
 \qquad
 \mathcal{T}[\nu_T] = \lambda\gamma_K(\nu_T,\mu_T)^2,
\end{equation}
where $W$ is symmetric positive definite, so its quadratic functional is
non-negative. Translation invariance gives the $O(dNM)$ random-feature
estimator of Section~\ref{sec:3.1interaction} ($O(NM)$ for fixed $d$).
The functional $\gamma_K$ is the MMD associated with a characteristic kernel $K$
that relaxes the deadline constraint $\nu_T=\mu_T$.
The kernel choice is motivated by the kernel-based Schr\"odinger bridge problem (SBP) studied by the author in~\cite{nak:2024sb},
which corresponds to the special case $\mathcal{R}\equiv 0$ of \eqref{eq:1.mfg_intro}--\eqref{eq:1.kernel_costs};
in that work the empirical kernel $U$-statistic of \eqref{eq:empirical_mmd} was used,
restricting the experiments to dimensions $d\le 2$ due to the $O(dN^2)$ kernel-evaluation cost.

\begin{equation}
\label{eq:empirical_mmd}
 \bar{\gamma}_K^2 = \frac{1}{N(N-1)}\sum_i\sum_{j\neq i} K(X_i,X_j) - \frac{2}{N^2}\sum_{i,j}K(X_i,Y_j) + \frac{1}{N(N-1)}\sum_i\sum_{j\neq i}K(Y_i,Y_j).
\end{equation}

To overcome this barrier, we use random Fourier features (RFF)
\cite{rah-rec:2007}. Fourier MMD estimators are established in two-sample
testing \cite{zha-men:2015,sut-sch:2015,cho-kim:2024}; in particular, FastMMD
\cite{zha-men:2015} includes the diagonal-corrected unbiased estimator.
We adapt that principle to the controlled-path objective: \eqref{eq:2.4}
yields an $O(N)$ $U$-statistic per frequency for the terminal MMD and, for
translation-invariant $W$, the running interaction
(Theorem~\ref{thm:rfu_interaction}). Our contributions are their joint use in
stochastic control, the $O(1/M)+O(1/N)$ variance decomposition, and the
population-penalty and independent held-out analyses as $M,N,\lambda$ vary.
The control-dependent EV cost instead has the $O(N)$ estimator
\eqref{eq:5ev.Rhat} and lies outside the population theorem.
The drift $u$ is parametrized by a neural network and trained by stochastic gradient descent
using a freshly resampled loss for \eqref{eq:1.mfg_intro};
the empirical interaction cost $\hat{\mathcal{R}}[\hat\nu_t^N]$ and the empirical terminal penalty
$\lambda\hat\gamma_{M,N}^2(\hat\nu_T^N,\mu_T)$ are evaluated on the current minibatch
at $O(dNM)$ cost (or $O(N)$ for the control-dependent aggregate-demand term in the EV extension).

\paragraph{Relation to existing computational MFG methods.}
Grid-based HJB--Fokker--Planck schemes
\cite{ach-cap:2010,ach-cam-cap:2012} are rigorous but dimension-limited.
Variational solvers include augmented-Lagrangian \cite{ben-car:2015alg2},
proximal and primal--dual splitting
\cite{bri-kal-sil:2018,bri-etal:2019,bri-der-lop-sil:2023,bri-sil-yan:2024},
and Frank--Wolfe methods
\cite{lav-pfe:2023,liu-pfe:2025,nakam-sai:2026loc,nakam-sai:2026disc}.
Their numerical implementations use a spatial discretization; \emph{mesh-independent} in
\cite{liu-pfe:2025} refers to convergence constants, not removal of the
underlying mesh. We instead sample controlled paths directly, while retaining
time discretization.

Fourier representations of nonlocal couplings appear in
\cite{nur-sau:2018,agr-etal:2022}, but without a terminal MMD penalty and
held-out analysis. The mesh-free blob method of Craig et al.
\cite{cra-ela-lee:2025} uses a mollified $L^2$ terminal penalty and proves
$\Gamma$-convergence in the successive limits $\varepsilon\to0$, $\delta\to0$,
and $N\to\infty$, together with diagonal convergence of minimizers; we use a
fixed characteristic kernel for a controlled diffusion and quantify random-feature
and sampling errors.

High-dimensional alternatives include APAC-Net and value-function networks
\cite{lin-etal:2021apac,rut-etal:2020}, fictitious-play and policy-iteration
methods \cite{car-lau:2022,car-lau:2021}, and bridge-matching methods for the
$c=0$ case \cite{de-etal:2021,var-etal:2021,shi-etal:2023dsbm,Tong2024}.
Within the kernel-MMD subclass, our single drift network and closed-form
$O(dNM)$ estimators avoid alternating population--control or auxiliary score
networks and yield the population and held-out guarantees of
Section~\ref{sec:4}. Experiments reach $d=100$ and include a
control-dependent EV extension.

\paragraph{Scope.}
The terminal cost is the squared MMD against $\mu_T$, relaxing
$\nu_T=\mu_T$ and admitting an unbiased $O(dNM)$ estimator. The population
theorem allows any non-negative weakly continuous running functional
$\mathcal R$; sample-level evaluation also requires a consistent estimator
(Proposition~\ref{prop:convergence_general}). This includes the kernel
interactions of Section~\ref{sec:3.1interaction}, whereas the
control-dependent EV cost of Section~\ref{sec:5ev} is a numerical extension
outside the theorem.
Extending to weakly lower-semicontinuous functionals such as local-density costs $\int F(\rho(x))\,dx$
is a natural next step.

\paragraph{Contributions.}
\begin{enumerate}[label=\textup{(\roman*)}]
\item Building on the diagonal-corrected Fourier estimator of FastMMD
\cite{zha-men:2015}, we derive a common $U$-statistic formulation for
$\gamma_K^2$ and the kernel self-interaction $\mathcal R$, giving estimators
$\hat\gamma_{M,N}^2$ for $\gamma_K^2$ (Theorem~\ref{thm:rfu_unbiased})
and $\hat{\mathcal{R}}_{M,N}[\hat\nu^N]$ for $\mathcal{R}[\nu]$ (Theorem~\ref{thm:rfu_interaction}),
each of complexity $O(dNM)$ (or $O(NM)$ for fixed $d$) and admitting an explicit variance decomposition $O(1/M)+O(1/N)$.
\item We formulate the kernel-MMD soft-target MFOT population objective and its sample loss
(Section~\ref{sec:4.1mfg}) and prove convergence of population penalized
near-minimizers to the terminal-constrained value (Theorem~\ref{thm:convergence}).
On an independent evaluation batch the empirical objective has the same almost-sure
limit when
$\lambda_n\sqrt{\log n/M_n},\lambda_n\sqrt{\log n/N_n}\to0$
(Corollary~\ref{cor:empirical_evaluation}).
Moreover,
$\gamma_K^2\le(H^*_c+\varepsilon_n)/\lambda_n$, giving
$\gamma_K=O(n^{-a/2})$ for $\lambda_n=n^a$ and bounded $\varepsilon_n$
(Propositions~\ref{prop:rate} and~\ref{prop:rate_mfg}).
\item The Schr{\"o}dinger bridge problem of~\cite{nak:2024sb} is recovered as the special case $\mathcal{R}\equiv 0$
(Corollary~\ref{cor:sbp_special_case}),
and taking the $M=\infty$ limit yields an independent sample-level evaluation theorem for the kernel-MMD penalty method
of~\cite{nak:2024sb} (Corollary~\ref{cor:kernel_convergence}).
The convergence theorem itself extends beyond the kernel-MMD structure of the running cost
to any non-negative weakly continuous functional of the state law with a consistent
empirical estimator (Proposition~\ref{prop:convergence_general}).
\item Numerical experiments cover the Schr{\"o}dinger bridge specialization in dimensions up to $d=100$
(Gaussian shift and bimodal targets; Sections~\ref{sec:5.2}--\ref{sec:5.5comp})
and a physically consistent EV charging extension with bounded per-vehicle
heterogeneity (Section~\ref{sec:5ev}), illustrating the trade-off between peak
normalized charging power and the terminal service level.
\end{enumerate}

The remainder of this paper is organized as follows.
Section~\ref{sec:2} recalls basic facts on the MMD and derives the Fourier representation
used by both the terminal MMD penalty and the kernel congestion cost.
Section~\ref{sec:3} presents the proposed unbiased $O(dNM)$ $U$-statistic estimators
and establishes their unbiasedness, complexity, and variance properties.
Section~\ref{sec:4} formulates the population potential MFG problem and its sample loss
and states the population convergence and independent sample-evaluation results,
together with the Schr{\"o}dinger-bridge specialization.
Section~\ref{sec:5} presents numerical experiments,
beginning with the Schr{\"o}dinger-bridge special case as a warm-up
and proceeding to the EV charging fleet coordination problem.
Section~\ref{sec:6} gives proofs of the main results.

\paragraph{Notation.}
Throughout the paper, $d\ge 1$ is a positive integer,
and $|\cdot|$ denotes the Euclidean norm on $\RR^d$. 
For any column $x\in\RR^d$, we write $x^{\mathsf{T}}$ for the transposition of $x$ 
and $x_k$ for the $k$-th component of $x$, $k=1,\ldots,d$.
For an $\RR^d$-valued random variable $X$, $\EE[X]_k$ denotes the $k$-th component
of its expectation vector $\EE[X]\in\RR^d$.
The set of Borel probability measures on $\RR^d$ is denoted by $\cP(\RR^d)$,
and we write $\mathcal{L}(X)\in\cP(\RR^d)$ for the distribution of an $\RR^d$-valued random variable $X$.
For a topological space $E$, $C_b(E)$ denotes the space of bounded continuous real-valued functions on $E$.
For an integrable function $\Phi:\RR^d\to\RR$,
$\widehat{\Phi}$ denotes its Fourier transform with the convention
$\widehat{\Phi}(\xi)=(2\pi)^{-d/2}\int_{\RR^d}e^{-\mathrm{i}\xi^{\mathsf{T}}x}\Phi(x)\,dx$,
where $\mathrm{i}=\sqrt{-1}$.
$I_k$ denotes the $k\times k$ identity matrix,
and we write $\mathcal{N}(m,\Sigma)$ for the Gaussian distribution on $\RR^k$ with mean $m\in\RR^k$ and covariance matrix $\Sigma$.
$\EE[\cdot]$ and $\mathrm{Var}(\cdot)$ stand for expectation and variance, respectively. 
For probability measures $Q$, $P$ on a common measurable space,
$H(Q\,|\,P):=\int\log(dQ/dP)\,dQ$ when $Q\ll P$, and $H(Q\,|\,P):=+\infty$ otherwise,
is the Kullback--Leibler divergence (relative entropy) of $Q$ with respect to $P$.
Finally, $C([0,T];\RR^d)$ denotes the space of continuous paths $[0,T]\to\RR^d$
endowed with the supremum norm.

\section{Hilbert space embeddings and Fourier representation of MMD}\label{sec:2}

\subsection{Maximum mean discrepancy}\label{sec:2.1}

Let $K\in C_b(\RR^d\times\RR^d)$ be a symmetric and strictly positive definite kernel on $\RR^d$, i.e.,
$K(x,y)=K(y,x)$ for $x,y\in\RR^d$ and for any pairwise distinct $x_1,\ldots,x_L\in \RR^d$ and
$\alpha=(\alpha_1,\ldots,\alpha_L)^{\mathsf{T}}\in\RR^L\setminus\{0\}$,
\[
 \sum_{j,\ell=1}^L\alpha_j\alpha_{\ell}K(x_j,x_{\ell})>0.
\]
Then there exists a unique Hilbert space $\cH\subset C_b(\RR^d)$ such that $K$ is a reproducing kernel on
$\cH$ with norm $\|\cdot\|$ (see, e.g., Wendland \cite{wen:2005});
the inclusion $\cH\subset C_b(\RR^d)$ follows from the reproducing property
$|f(x)|=|\langle f,K(\cdot,x)\rangle|\le \|f\|\sqrt{K(x,x)}\le \|f\|\sup_z\sqrt{K(z,z)}$,
which is finite since $K\in C_b(\RR^d\times\RR^d)$.
The \textit{maximum mean discrepancy} (MMD) associated with $K$ is defined by
\begin{equation}
\label{eq:2.1}
 \gamma_K(\mu,\nu): = \sup_{f\in\cH, \, \|f\|\le 1}\left|\int_{\RR^d} f d\mu - \int_{\RR^d} f d\nu\right|,
 \quad \mu,\nu\in\cP(\RR^d)
\end{equation}
(see Gretton et al.~\cite{gre-etal:2006}).
We assume that $\gamma_K$ defines a metric on $\cP(\RR^d)$, i.e., $K$ is a \textit{characteristic kernel}.
By the reproducing property,
\begin{equation}
\label{eq:2.2}
 \gamma_K(\mu,\nu)^2 = \iint\nolimits_{\RR^d\times\RR^d}K(x,y)(\mu-\nu)(dx)(\mu-\nu)(dy)
\end{equation}
(see Section 2 in Sriperumbudur et al.~\cite{sri-etal:2010}).

\subsection{Fourier representation}\label{sec:2.2}

We now derive the Fourier representation that is central to our approach.
Consider the case where
\begin{enumerate}
\item[(A)] $K$ is represented as $K(x,y)=\Phi(x-y)$, $x,y\in\RR^d$, for some bounded, continuous and integrable function $\Phi$ on $\RR^d$ such that
 its Fourier transform $\widehat{\Phi}$ is also integrable on $\RR^d$ and satisfies $\widehat{\Phi}(\xi)>0$ for any $\xi\in\RR^d$.
\end{enumerate}
Under (A), the kernel $K$ is in $C_b(\RR^d\times\RR^d)$ that is symmetric and strictly positive definite.  
Moreover, the Fourier inversion formula gives
\begin{equation}
\label{eq:2.3}
\Phi(x) = \int_{\RR^d}\rho(\xi)e^{\mathrm{i}\xi^{\mathsf{T}}x}d\xi, \quad x\in\RR^d,
\end{equation}
where $\rho(\xi)=(2\pi)^{-d/2}\widehat{\Phi}(\xi)>0$ for $\xi\in\RR^d$.
When several kernels are in play we write $\rho_K:=\rho$ for the spectral density of $K$,
and $\rho_W$ for that of the interaction kernel $W$ of Section~\ref{sec:3.1interaction}.
Substituting \eqref{eq:2.3} into \eqref{eq:2.2} and applying Fubini's theorem, we obtain
\begin{equation}
\label{eq:2.4}
\gamma_K(\mu,\nu)^2 = \int_{\RR^d}|\tilde{\mu}(\xi) - \tilde{\nu}(\xi)|^2 \rho(\xi) d\xi,
\end{equation}
where $\tilde{\mu}(\xi):=\int_{\RR^d}e^{\mathrm{i}\xi^{\mathsf{T}}x}\mu(dx)$ denotes the characteristic function of $\mu$.

Note that $\rho$ integrates to $\Phi(0)=K(0,0)$.
Let $X\sim\mu$, $Y\sim \nu$, $X^{\prime}\sim\mu$, and $Y^{\prime}\sim\nu$, with
$X,X^{\prime},Y,Y^{\prime}$ mutually independent.
Then, since $|\tilde{\mu}(\xi) - \tilde{\nu}(\xi)|^2$ is real-valued,
\begin{align}
|\tilde{\mu}(\xi) - \tilde{\nu}(\xi)|^2
&= \EE[e^{\mathrm{i}\xi^{\mathsf{T}}X} - e^{\mathrm{i}\xi^{\mathsf{T}}Y}]\overline{\EE[e^{\mathrm{i}\xi^{\mathsf{T}}X} - e^{\mathrm{i}\xi^{\mathsf{T}}Y}]} \notag \\
&= \EE\left[e^{\mathrm{i}\xi^{\mathsf{T}}(X-X^{\prime})} - e^{\mathrm{i}\xi^{\mathsf{T}}(X-Y^{\prime})} - e^{\mathrm{i}\xi^{\mathsf{T}}(Y-X^{\prime})} + e^{\mathrm{i}\xi^{\mathsf{T}}(Y-Y^{\prime})}\right] \notag \\
\label{eq:2.5}
&= \EE[\cos(\xi^{\mathsf{T}}(X-X^{\prime}))] - 2\EE[\cos(\xi^{\mathsf{T}}(X-Y))] + \EE[\cos(\xi^{\mathsf{T}}(Y-Y^{\prime}))],
\end{align}
where in the last step we used the fact that $(X,Y')$ and $(X,Y)$ have the same joint distribution $\mu\otimes\nu$,
and similarly $\EE[\cos(\xi^{\mathsf{T}}(Y-X'))] = \EE[\cos(\xi^{\mathsf{T}}(X-Y))]$ by the symmetry of $\cos$.
Now, let $Z$ be a random variable with density $\rho/\Phi(0)$, independent of $(X,X',Y,Y')$.
Integrating \eqref{eq:2.5} against $\rho(\xi)d\xi$ and using \eqref{eq:2.4}, we arrive at the representation
\begin{equation}
\label{eq:repre}
 \gamma_K(\mu,\nu)^2 = \Phi(0)\,\EE\left[\cos(Z^{\mathsf{T}}(X-X^{\prime})) - 2 \cos(Z^{\mathsf{T}}(X-Y)) + \cos(Z^{\mathsf{T}}(Y-Y^{\prime})) \right].
\end{equation}

\begin{rem}\label{rem:2.1}
The representation \eqref{eq:repre} also shows that $\gamma_K$ metrizes the weak topology on $\cP(\RR^d)$ under $(A)$.
See \cite[Proposition 2.1]{nak:2024sb} for a proof.
\end{rem}

\begin{ex}\label{ex:2.2}
The Gaussian kernel $K(x,y)=e^{-\alpha|x-y|^2}$, $\alpha>0$, satisfies assumption $(A)$. 
Indeed, $K(x,y)=\Phi(x-y)$ with $\Phi(z)=e^{-\alpha|z|^2}$ bounded, continuous, and integrable on $\RR^d$,
and its inverse Fourier transform
$\rho(\xi)=(2\pi)^{-d/2}\widehat\Phi(\xi)=(4\pi\alpha)^{-d/2}e^{-|\xi|^2/(4\alpha)}>0$
is exactly the density of $\mathcal{N}(0,2\alpha I_d)$, since $\Phi(0)=1$ here.
Hence $\Phi(0)=1$ and the random frequency in \eqref{eq:repre} satisfies $Z\sim \mathcal{N}(0,2\alpha I_d)$.
\end{ex}

\section{Unbiased estimator via random Fourier $U$-statistics}\label{sec:3}

For the estimator constructions in this section we fix a complete probability space
$(\Omega,\mathcal{F},\PP)$ and denote its expectation by $\EE$.
The weak control systems introduced in Section~\ref{sec:4} may live on different probability spaces.

\subsection{Construction of the estimator}\label{sec:3.1}

The representation \eqref{eq:repre} suggests a Monte Carlo approach:
draw $Z_1,\ldots,Z_M$ i.i.d.\ from the density $\rho/\Phi(0)$ and approximate $\gamma_K^2$ by averaging over the random frequencies.
Given i.i.d.\ samples $X_1,\ldots,X_N\sim\mu$ and $Y_1,\ldots,Y_N\sim\nu$,
we define for each $z\in\RR^d$ the following quantities:
\begin{equation}
\label{eq:3.1}
 S_c^X(z) := \sum_{i=1}^N\cos(z^{\mathsf{T}}X_i), \quad
 S_s^X(z) := \sum_{i=1}^N\sin(z^{\mathsf{T}}X_i),
\end{equation}
and similarly $S_c^Y(z)$, $S_s^Y(z)$ for the $Y$-samples.

The key observation is that the $U$-statistic for $\EE[\cos(z^{\mathsf{T}}(X-X'))]$ can be computed in $O(N)$:
\begin{equation}
\label{eq:3.2}
 U_{XX}(z) := \frac{(S_c^X(z))^2 + (S_s^X(z))^2 - N}{N(N-1)}.
\end{equation}
Indeed, expanding the squares and using the trigonometric identity $\cos(a-b) = \cos a\cos b + \sin a\sin b$, we get 
\begin{align*}
 (S_c^X)^2 + (S_s^X)^2
 &= \sum_{i,j=1}^N (\cos(z^{\mathsf{T}}X_i)\cos(z^{\mathsf{T}}X_j) + \sin(z^{\mathsf{T}}X_i)\sin(z^{\mathsf{T}}X_j)) \\
 &= \sum_{i,j=1}^N \cos(z^{\mathsf{T}}(X_i - X_j)) = N + \sum_{i\neq j}\cos(z^{\mathsf{T}}(X_i - X_j)).
\end{align*}
Dividing by $N(N-1)$ yields an unbiased estimator of $\EE[\cos(z^{\mathsf{T}}(X-X'))]$.

Similarly, the cross term $\EE[\cos(z^{\mathsf{T}}(X-Y))]$ is estimated by
\begin{equation}
\label{eq:3.3}
 V_{XY}(z) := \frac{S_c^X(z)\,S_c^Y(z) + S_s^X(z)\,S_s^Y(z)}{N^2},
\end{equation}
which is unbiased since $X_i$ and $Y_j$ are always from independent samples.

Combining the representation \eqref{eq:repre} with the Monte Carlo approach,
we propose the \emph{RF $U$-statistic estimator}: given i.i.d.\ random frequencies
$Z_1,\ldots,Z_M\sim\rho/\Phi(0)$ independent of the data, set
\begin{equation}
\label{eq:3.4}
 \hat{\gamma}_{M,N}^2 := \frac{\Phi(0)}{M}\sum_{r=1}^M\left[U_{XX}(Z_r) - 2V_{XY}(Z_r) + U_{YY}(Z_r)\right].
\end{equation}

The following theorem collects the basic properties of our penalty term. 
The proof is given in Section~\ref{sec:6.1}.

\begin{thm}[Unbiasedness, complexity, and variance]\label{thm:rfu_unbiased}
Suppose $(A)$ holds.
Let $X_1,\ldots,X_N$ with $N\ge 2$, $Y_1,\ldots,Y_N$, and $Z_1,\ldots,Z_M$ with $M\ge 1$
be three mutually independent collections of i.i.d.\ random variables
with $X_i\sim\mu$, $Y_j\sim\nu$, and $Z_r\sim\rho/\Phi(0)$.
Then the following hold.
\begin{enumerate}[label=\textup{(\roman*)}]
\item $\EE[\hat{\gamma}_{M,N}^2] = \gamma_K(\mu,\nu)^2$.
\item The computational complexity of $\hat{\gamma}_{M,N}^2$ is $O(dNM)$, and hence $O(NM)$ when $d$ is fixed.
\item $\mathrm{Var}(\hat{\gamma}_{M,N}^2) = O(1/M) + O(1/N)$.
\end{enumerate}
\end{thm}

\begin{rem}\label{rem:3.2}
The standard RFF approach computes a $V$-statistic
$|\frac{1}{N}\sum_i \phi(X_i) - \frac{1}{N}\sum_j \phi(Y_j)|^2$ with positive bias of order $\Phi(0)/N$ when $\mu=\nu$.
Our $\hat{\gamma}_{M,N}^2$ eliminates this bias while preserving sample-linear $O(dNM)$ complexity,
which avoids a non-vanishing penalty floor at the population level
and yields the clean variance decomposition of Theorem~{\rm \ref{thm:rfu_unbiased}}(iii)
and the explicit a.s.\ rate of Proposition~{\rm \ref{prop:rate}} below
without any bias-correction term.
Empirically the bias is mild near the optimum. 
Section~\ref{sec:5.1.3vs} shows that the trained drifts from the biased $V$-statistic and the unbiased $U$-statistic
are indistinguishable across the tested range of $(N,\lambda)$.
The unbiased construction is thus preferred for theoretical reasons rather than for an immediate empirical accuracy gain.
\end{rem}

\subsection{Estimator for the kernel interaction cost}\label{sec:3.1interaction}

The same Fourier $U$-statistic construction extends to the running interaction cost $\mathcal{R}[\nu]$ defined by \eqref{eq:4.cost_decomp}.
Suppose $(A)$ holds for the congestion kernel $W(x,y) = \Psi(x-y)$,
with generating density $\rho_W(\xi):=(2\pi)^{-d/2}\widehat{\Psi}(\xi)>0$.
The Fourier inversion formula and Fubini's theorem (cf.\ Section~\ref{sec:2.2}) then give
\begin{equation}\label{eq:repre_R}
 \iint\nolimits_{\RR^d\times\RR^d} W(x,y)\,\nu(dx)\nu(dy)
 = \Psi(0)\,\EE\bigl[\cos(\tilde Z^{\mathsf{T}}(X-X'))\bigr],
\end{equation}
where $\tilde Z\sim \rho_W/\Psi(0)$ and $X,X'$ is an independent pair of samples from $\nu$,
all mutually independent.
Given i.i.d.\ samples $X_1,\ldots,X_N\sim\nu$ ($N\ge 2$)
and i.i.d.\ random frequencies $\tilde Z_1,\ldots,\tilde Z_M\sim\rho_W/\Psi(0)$ ($M\ge 1$),
independent of the data, define
\begin{equation}\label{eq:3.5}
 \hat{\mathcal{R}}_{M,N}[\hat\nu^N]
 := \frac{\Psi(0)}{2M}\sum_{r=1}^M U_{XX}(\tilde Z_r)
 = \frac{\Psi(0)}{2M}\sum_{r=1}^M\frac{S_c^X(\tilde Z_r)^2+S_s^X(\tilde Z_r)^2-N}{N(N-1)},
\end{equation}
where $U_{XX}$, $S_c^X$, $S_s^X$ are as in \eqref{eq:3.1}--\eqref{eq:3.2}
evaluated at the single sample collection $\{X_i\}_{i=1}^N$.
The estimator \eqref{eq:3.5} is the natural analogue of $\hat\gamma_{M,N}^2$ for the
self-interaction kernel functional $\frac{1}{2}\iint W\,d\nu\otimes d\nu$:
it removes the diagonal bias of the $V$-statistic via the same trigonometric identity
and is also evaluable in $O(dNM)$.

\begin{thm}[Properties of the interaction estimator]\label{thm:rfu_interaction}
Suppose $(A)$ holds for $W$.
Let $X_1,\ldots,X_N$ with $N\ge 2$ and $\tilde Z_1,\ldots,\tilde Z_M$ with $M\ge 1$
be two mutually independent collections of i.i.d.\ random variables
with $X_i\sim\nu$ and $\tilde Z_r\sim\rho_W/\Psi(0)$. Then
\begin{enumerate}[label=\textup{(\roman*)}]
\item $\EE\bigl[\hat{\mathcal{R}}_{M,N}[\hat\nu^N]\bigr] = \mathcal{R}[\nu]$.
\item The computational complexity of $\hat{\mathcal{R}}_{M,N}[\hat\nu^N]$ is $O(dNM)$, and hence $O(NM)$ when $d$ is fixed.
\item $\mathrm{Var}\bigl(\hat{\mathcal{R}}_{M,N}[\hat\nu^N]\bigr) = O(1/M) + O(1/N)$.
\end{enumerate}
\end{thm}

\begin{proof}
The proof is parallel to that of Theorem~\ref{thm:rfu_unbiased}
and is given in Section~\ref{sec:6.1interaction}.
\end{proof}

\section{Potential mean-field game formulation and convergence analysis}\label{sec:4}

\subsection{Potential MFG with kernel costs}\label{sec:4.1mfg}

Fix the horizon $T>0$, the diffusion coefficient $\sigma>0$,
the initial law $\mu_0\in\cP(\RR^d)$ with finite second moment,
and the \emph{deadline target} $\mu_T\in\cP(\RR^d)$ with finite second moment.
Here $\mu_T$ prescribes the distribution that the controlled population is required
to attain at the terminal time $t=T$,
and the deviation $\mathcal{L}(X_T)$ from $\mu_T$ is enforced softly via a kernel-MMD penalty
(introduced in \eqref{eq:4.1mfg} below).
We fix a reproducing kernel $K:\RR^d\times\RR^d\to\RR$ satisfying assumption $(A)$
with generating function $\Phi$, the terminal penalty weight $\lambda>0$, and the congestion weight $c\ge 0$.
The running mean-field cost $\mathcal{R}:\cP(\RR^d)\to[0,\infty)$ is a non-negative weakly continuous functional,
specified below as a self-interaction integral with a kernel $W$ satisfying $(A)$
(Section~\ref{sec:3.1interaction} and \eqref{eq:4.cost_decomp} below).

The controlled state equation is
\begin{equation}\label{eq:4.0}
 dX_t = u(t,X_t)\,dt + \sigma\,dB_t,\quad X_0\sim\mu_0,\quad t\in[0,T].
\end{equation}
We define the class of \emph{admissible controls} $\mathcal{A}$
as the set of pairs $u=(\mathsf{u},Q^u)$ consisting of a Borel measurable function
$\mathsf{u}:[0,T]\times\RR^d\to\RR^d$ and the law $Q^u$ on $C([0,T];\RR^d)$
of a weak solution of \eqref{eq:4.0} with drift $\mathsf{u}$ and initial law $\mu_0$,
such that $\EE_{Q^u}\bigl[\int_0^T|\mathsf{u}(t,X_t)|^2 dt\bigr]<\infty$.
We write $u$ for the pair and $\mathcal{L}(X^u):=Q^u$, and we suppress $\mathsf{u}$ from the notation;
carrying the solution law along with the drift avoids any selection issue
when weak uniqueness fails, and every objective below depends on $u$ only through $Q^u$.
Any Borel function $u$ that is bounded and Lipschitz in $x$ (uniformly in $t$) belongs to $\mathcal{A}$,
with a unique strong solution, and so does any bounded Borel $u$ \cite{ver:1981};
allowing weak solutions is what makes $\mathcal{A}$ closed under the Markov projection of Lemma~\ref{lem:feedback}.

For the analysis of the terminal-constrained problem we also need the corresponding class of controls
in the weak sense, following \cite[Section~5]{nak:2024sb} and \cite{fol:1985};
see also \cite{fle-son:2006} and \cite{yon-zho:1999} for weak formulations of stochastic control problems.
\begin{defn}[Admissible control system]\label{def:acs}
A collection
\[
 \pi=(\Omega^\pi,\mathcal F^\pi,\{\mathcal F_t^\pi\}_{t\in[0,T]},
 \mathbb Q^\pi,B^\pi,u^\pi,Y^\pi)
\]
is called an \emph{admissible control system} if
\begin{enumerate}[label=\textup{(\roman*)}]
\item $(\Omega^\pi,\mathcal F^\pi,\{\mathcal F_t^\pi\},\mathbb Q^\pi)$ is a filtered probability space satisfying the usual conditions;
\item $B^\pi$ is a $d$-dimensional $\{\mathcal{F}_t^\pi\}$-Brownian motion under $\mathbb{Q}^\pi$;
\item $u^\pi=\{u^\pi_t\}_{t\in[0,T]}$ is an $\RR^d$-valued $\{\mathcal{F}_t^\pi\}$-progressively measurable process with
$\EE_{\mathbb{Q}^\pi}\bigl[\int_0^T|u^\pi_t|^2dt\bigr]<\infty$;
\item $Y^\pi=\{Y_t^\pi\}_{t\in[0,T]}$ is an $\RR^d$-valued continuous $\{\mathcal{F}_t^\pi\}$-adapted process such that
$\mathbb{Q}^\pi\circ (Y_0^\pi)^{-1}=\mu_0$ and
$Y_t^\pi=Y_0^\pi+\int_0^tu^\pi_s\,ds+\sigma B^\pi_t$ for $t\in[0,T]$.
\end{enumerate}
We write $\Pi$ for the set of all admissible control systems,
$\Pi_{\mu_T}:=\{\pi\in\Pi:\mathbb{Q}^\pi\circ (Y_T^\pi)^{-1}=\mu_T\}$, and
\[
 J_c(\pi):=\frac{1}{2}\EE_{\mathbb{Q}^\pi}\Bigl[\int_0^T|u^\pi_t|^2\,dt\Bigr]
 + c\int_0^T\mathcal{R}\bigl[\mathbb{Q}^\pi\circ (Y_t^\pi)^{-1}\bigr]\,dt .
\]
\end{defn}
Every $u\in\mathcal{A}$, together with any weak solution realizing its recorded law $Q^u$,
induces an element $\pi_u\in\Pi$ with
$J_c(\pi_u)=\frac{1}{2}\EE_{Q^u}[\int_0^T|u(t,X_t)|^2dt]+c\int_0^T\mathcal{R}[Q^u_t]\,dt$.
Thus the feedback class is embedded in $\Pi$ without requiring all weak solutions to live on one fixed probability space.
The value of the terminal-constrained potential MFG is defined in this weak sense as
\begin{equation}\label{eq:4.Hstar}
 H^*_c\;:=\;\inf_{\pi\in\Pi_{\mu_T}}J_c(\pi).
\end{equation}
The assumption $H_c^*<\infty$ used below is a genuine finite-energy feasibility
condition.  Indeed, Lemma~\ref{lem:entropy_energy}(i) and data processing imply that
the terminal law of every finite-energy admissible system is absolutely continuous
with respect to the terminal law of the uncontrolled diffusion.  Since $\sigma>0$,
the latter has a density; in particular, a Dirac terminal target has $H_c^*=\infty$
and is not covered by Theorem~\ref{thm:convergence}.
The enlargement is needed only for the lower bound in the convergence proof:
weak limits of laws of controlled processes have progressively measurable, and not necessarily feedback, drifts;
see Step~$(\mathrm{ii})$ in Section~\ref{sec:6.2}.
Since $\mathcal{A}\subset\Pi$ in the above sense, $H^*_c$ is not larger than the infimum over feedback controls.

The population penalized objective and its value are
\begin{equation}\label{eq:4.1mfg}
 \mathcal{J}_{\lambda,c}(u):=
 \frac{1}{2}\EE_{Q^u}\left[\int_0^T |u(t,X_t)|^2 dt\right]
 + c\int_0^T\mathcal{R}\bigl[\mathcal{L}(X_t)\bigr]\,dt
 + \lambda \mathcal{T}\bigl[\mathcal{L}(X_T)\bigr],
 \qquad H^*_{\lambda,c}:=\inf_{u\in\mathcal A}\mathcal J_{\lambda,c}(u),
\end{equation}
where the kernel-based interaction and terminal costs are
\begin{equation}\label{eq:4.cost_decomp}
 \mathcal{R}[\nu] = \frac{1}{2}\iint\nolimits_{\RR^d\times\RR^d}W(x,y)\,\nu(dx)\,\nu(dy),
 \qquad
 \mathcal{T}[\nu] = \gamma_K(\nu,\mu_T)^2.
\end{equation}
The interaction $\mathcal{R}$ is a quadratic functional of the population law and admits the linear functional derivative
$\delta\mathcal{R}[\nu]/\delta\nu(x) = \int W(x,y)\,\nu(dy)$,
making \eqref{eq:4.1mfg} a potential MFG in the sense of Lasry--Lions~\cite{las-lio:2007};
the terminal cost $\mathcal{T}[\nu]$ has linear functional derivative
$\delta\mathcal{T}[\nu]/\delta\nu(x) = 2\bigl(\int K(x,y)\nu(dy)-\int K(x,y)\mu_T(dy)\bigr)$
and equals zero exactly when $\nu=\mu_T$.

To make the connection with the game formulation explicit, fix a candidate
population flow $(\nu_t)_{0\le t\le T}$.  The associated representative agent
has running interaction cost
$c\,\delta\mathcal R[\nu_t]/\delta\nu(X_t)$ and terminal cost
$\lambda\,\delta\mathcal T[\nu_T]/\delta\nu(X_T)$, in addition to the kinetic
cost $|u|^2/2$.  The first-order optimality system for the potential problem is
therefore the representative-agent optimality system together with the
consistency condition $\nu_t=\mathcal L(X_t)$.  Under the standard regularity
and convexity assumptions that make these optimality conditions sufficient, a
minimizer of the potential induces an MFG equilibrium.  Conversely, without
additional convexity, satisfaction of the associated optimality system does
not by itself imply global minimization of the potential.

\begin{rem}[Behaviour of $\mathcal{R}$]\label{rem:R_behaviour}
Under $(A)$, the Bochner representation gives
\[
 \mathcal{R}[\nu] = \frac{1}{2}\int_{\RR^d}|\tilde\nu(\xi)|^2\rho_W(\xi)\,d\xi,
\]
where $\tilde\nu$ denotes the characteristic function of $\nu$.
In particular $\mathcal{R}$ takes values in $[0,\Psi(0)/2]$, is translation invariant,
and is monotone in the concentration of $\nu$. 
The maximum $\Psi(0)/2$ is attained at Dirac masses and $\mathcal{R}[\nu]\to 0$ as $\nu$ spreads.
For the Gaussian $W(x,y)=e^{-\alpha|x-y|^2}$ and $\nu=\mathcal{N}(m,s^2 I_d)$ with $s>0$,
$\mathcal{R}[\nu]=\frac{1}{2}(1+4\alpha s^2)^{-d/2}$.
The bandwidth $\alpha$ thus controls the scale at which population concentration is penalized:
the running cost $c\int_0^T\mathcal{R}[\mathcal{L}(X_t)]\,dt$ makes the optimal drift
spread agents apart on the scale $\alpha^{-1/2}$.
\end{rem}

At each iteration the cost in \eqref{eq:4.1mfg} is estimated from $N$ i.i.d.\ copies
$X^{(1)},\ldots,X^{(N)}$ of the controlled process,
$N$ i.i.d.\ target samples $Y^{(1)},\ldots,Y^{(N)}\sim\mu_T$,
and $M$ random frequencies for each kernel
(all jointly independent; see Algorithm~\ref{alg:sbp} for the precise schedule).
Write $\hat\nu^N_t := \frac{1}{N}\sum_{i=1}^N\delta_{X^{(i)}_t}$
for the empirical distribution at time $t$.
The empirical kernel costs are then
\begin{align}
\label{eq:4.empF}
 \hat{\mathcal{R}}_{M,N}[\hat\nu^N_t]
 &:= \frac{\Psi(0)}{M}\sum_{r=1}^M\frac{S_c^{X_t}(\tilde Z_r)^2 + S_s^{X_t}(\tilde Z_r)^2 - N}{2N(N-1)},\\
\label{eq:4.empG}
 \hat{\mathcal{T}}_{M,N}[\hat\nu^N_T,\mu_T]
 &:= \hat\gamma_{M,N}^2\bigl(\hat\nu^N_T,\mu_T\bigr),
\end{align}
where $\hat\gamma_{M,N}^2$ is the RF $U$-statistic estimator of Section~\ref{sec:3}
applied to the terminal samples,
and $S_c^{X_t}, S_s^{X_t}$ are the trigonometric sums of \eqref{eq:3.1} evaluated at the $N$ controlled paths at time $t$.
Throughout, we write $\hat\gamma_{M,N}^2(\nu,\mu)$ for the estimator computed from
$N$ i.i.d.\ samples of $\nu$, $N$ i.i.d.\ samples of $\mu$ and $M$ random frequencies:
the arguments record the sampling laws and not a deterministic function of the measures,
so that $\hat\gamma_{M,N}^2(\mathcal{L}(X_T),\mu_T)$ is the value computed from the terminal values
of the controlled paths together with the target samples, and coincides with \eqref{eq:4.empG}.
For a fixed control $u$, the corresponding sample loss is
\begin{equation}\label{eq:4.2}
 \hat{\mathcal J}_{\lambda,c,M,N}(u)
 :=
 \frac{1}{2N}\sum_{i=1}^N\int_0^T|u(t,X^{(i)}_t)|^2 dt
 + c\int_0^T\hat{\mathcal{R}}_{M,N}[\hat\nu^N_t]\,dt
 + \lambda\,\hat{\mathcal{T}}_{M,N}[\hat\nu^N_T,\mu_T].
\end{equation}
Equation~\eqref{eq:4.2} is used either as a fresh stochastic loss during training
or to evaluate a fixed, possibly data-dependent control on an independent batch.
We do not take the infimum of a single fixed-sample functional over the full weak
class $\mathcal A$: the convergence theorem concerns population near-minimizers,
and Algorithm~\ref{alg:sbp} (Section~\ref{sec:4.3alg}) draws a fresh independent
realization at each gradient step (see Remark~\ref{rem:3.3}).

\begin{rem}[Schr{\"o}dinger-bridge specialization]\label{rem:sbp_specialization}
Setting $c=0$ and $T=1$ in \eqref{eq:4.1mfg} reduces the potential MFG to the kernel-MMD-penalty Schr{\"o}dinger bridge of \cite{nak:2024sb}, namely
\begin{equation}\label{eq:4.1sbp}
 H^*_{\lambda,0}\bigr|_{T=1}
 = \inf_{u\in\mathcal{A}}\left\{\frac{1}{2}\EE\left[\int_0^1|u(t,X_t)|^2\,dt\right]
 + \lambda\,\gamma_K(\mathcal{L}(X_1),\mu_1)^2\right\}.
\end{equation}
The corresponding convergence theorem is Corollary~{\rm \ref{cor:sbp_special_case}} below.
\end{rem}

\begin{rem}\label{rem:weak_vs_strong}
The classical mean-field game and Schr{\"o}dinger bridge problems are often formulated in the \emph{weak} sense,
where the probability space and the driving Brownian motion are themselves part of the optimization
(see, e.g., \cite{leo:2014}).
Since Theorem~{\rm \ref{thm:convergence}} depends only on the laws and costs of the controlled processes,
the feedback formulation based on $\mathcal{A}$ adopted here is sufficient
and matches the implementation in Algorithm~\ref{alg:sbp},
where $u$ is parametrized as a neural network feedback;
the weak class $\Pi$ of Definition~{\rm \ref{def:acs}} enters only through the value $H^*_c$
and through the Markov projection of Lemma~{\rm \ref{lem:feedback}}.
\end{rem}

\subsection{Population convergence and independent sample evaluation}\label{sec:4.2conv}

Before stating the main theorem, we record two consistency facts used for sample-level evaluation.

\begin{lem}[Strong consistency of the MMD estimator]\label{lem:strong_consistency}
Suppose that $(A)$ holds.
Let $X_1,X_2,\ldots\sim\mu$, $Y_1,Y_2,\ldots\sim\nu$, and
$Z_1,Z_2,\ldots\sim\rho/\Phi(0)$ be three mutually independent i.i.d.\ sequences,
and construct the estimator at stage $n$ from the first $N_n,N_n$, and $M_n$ variables, respectively.
Let $\{M_n\}_{n\ge 1}$ and $\{N_n\}_{n\ge 1}$ be deterministic sequences with $N_n\to\infty$
and $M_n/\log n\to\infty$.
Then
\[
 \hat{\gamma}_{M_n,N_n}^2 \longrightarrow \gamma_K(\mu,\nu)^2 \quad n\to\infty, \:\;\text{a.s.}
\]
Alternatively, for each $n$, let
$X^{(n)}_1,\ldots,X^{(n)}_{N_n}$,
$Y^{(n)}_1,\ldots,Y^{(n)}_{N_n}$, and
$Z^{(n)}_1,\ldots,Z^{(n)}_{M_n}$ be three mutually independent i.i.d.\ samples
with laws $\mu$, $\nu$, and $\rho/\Phi(0)$, respectively, and construct the
stage-$n$ estimator from these samples.  Then the same conclusion holds under
the additional condition $N_n/\log n\to\infty$.
\end{lem}

\begin{proof}
See Section~\ref{sec:6.LemSC}.
\end{proof}

Since $\mathcal{C}(u)<\infty$ for $u\in\mathcal{A}$, Kolmogorov's strong law of large numbers yields the following, with no moment assumption beyond the first: 
\begin{lem}[Strong consistency of the empirical energy]\label{lem:energy_consistency}
Fix a deterministic feedback control $u\in\mathcal{A}$ and let $X^{(1)},X^{(2)},\ldots$ be
i.i.d.\ copies of the controlled process $X^u$ of \eqref{eq:4.0}.
Then
\[
 \frac{1}{N}\sum_{i=1}^N \int_0^T |u(t,X_t^{(i)})|^2 dt
 \;\longrightarrow\; \EE_{Q^u}\!\left[\int_0^T |u(t,X_t)|^2 dt\right], \quad N\to\infty, \;\;\text{a.s.},
\]
by Kolmogorov's strong law of large numbers, since the summands are i.i.d.\ and integrable.
\end{lem}

When the solution is strong, these copies are obtained from independent copies of the driving
noise and initial state; in general, they are independent copies of the solution law.

Put
\[
 \mathcal C(u):=\EE_{Q^u}\!\left[\int_0^T|u(t,X_t)|^2dt\right],
 \qquad
 \mathcal R_T(u):=\int_0^T\mathcal R[\mathcal L(X_t)]dt .
\]
We shall use the following entropy--energy fact.  It separates the two directions that
are sometimes both referred to as Girsanov's theorem.

\begin{lem}[Entropy--energy inequality and finite-entropy representation]\label{lem:entropy_energy}
Let $P^\sigma$ be the law on $C([0,T];\RR^d)$ of
$X_0+\sigma B$, where $X_0\sim\mu_0$ is independent of $B$.
\begin{enumerate}[label=\textup{(\roman*)}]
\item If $\pi\in\Pi$ has path law $Q^\pi:=\mathbb Q^\pi\circ(Y^\pi)^{-1}$, then
\[
 \sigma^2 H(Q^\pi\mid P^\sigma)
 \le \frac12\EE_{\mathbb Q^\pi}\!\left[\int_0^T|u_t^\pi|^2dt\right].
\]
\item Conversely, if $Q\in\cP(C([0,T];\RR^d))$ satisfies $Q_0=\mu_0$ and
$H(Q\mid P^\sigma)<\infty$, then there exists $\pi^Q\in\Pi$ with path law $Q$ and
\[
 \frac12\EE_{\mathbb Q^{\pi^Q}}\!\left[\int_0^T|u_t^{\pi^Q}|^2dt\right]
 =\sigma^2H(Q\mid P^\sigma).
\]
\end{enumerate}
\end{lem}

\begin{proof}
For part~\textup{(i)}, set
\[
 \tau_m:=\inf\left\{t\in[0,T]:\int_0^t|u_s^\pi|^2ds\ge m\right\}\wedge T,
 \qquad u_t^{\pi,m}:=\bm{1}_{\{t\le\tau_m\}}u_t^\pi,
\]
and define
\[
 Y_t^{\pi,m}:=Y_0^\pi+\int_0^tu_s^{\pi,m}ds+\sigma B_t^\pi.
\]
The stochastic exponential
\[
 Z_m:=\exp\left(
 -\frac1\sigma\int_0^Tu_s^{\pi,m}\cdot dB_s^\pi
 -\frac1{2\sigma^2}\int_0^T|u_s^{\pi,m}|^2ds
 \right)
\]
is a true $\mathbb Q^\pi$-martingale by Novikov's condition, because the stopped
quadratic variation is bounded by $m/\sigma^2$.  Define a probability measure
$\mathbb P_m$ on $(\Omega^\pi,\mathcal F^\pi)$ by
$d\mathbb P_m=Z_m\,d\mathbb Q^\pi$.  Girsanov's theorem gives that
\[
 B_t^{\pi,m}:=B_t^\pi+\frac1\sigma\int_0^tu_s^{\pi,m}ds
\]
is a Brownian motion under $\mathbb P_m$.  Moreover,
$\mathbb E_{\mathbb Q^\pi}[Z_m\mid\mathcal F_0^\pi]=1$, so
$\mathbb P_m\circ(Y_0^\pi)^{-1}=\mu_0$, and hence
\[
 Y_t^{\pi,m}=Y_0^\pi+\sigma B_t^{\pi,m},
 \qquad \mathbb P_m\circ(Y^{\pi,m})^{-1}=P^\sigma.
\]
Since $Z_m>0$, $\mathbb Q^\pi\ll\mathbb P_m$, and the stopped stochastic integral
has zero $\mathbb Q^\pi$-expectation.  Therefore
\[
 H(\mathbb Q^\pi\mid\mathbb P_m)
 =\frac1{2\sigma^2}\mathbb E_{\mathbb Q^\pi}
   \left[\int_0^T|u_s^{\pi,m}|^2ds\right].
\]
The data-processing inequality for relative entropy, applied to the map
$Y^{\pi,m}$, now yields
\[
 H\!\left(\mathbb Q^\pi\circ(Y^{\pi,m})^{-1}\mid P^\sigma\right)
 \le \frac1{2\sigma^2}\mathbb E_{\mathbb Q^\pi}
   \left[\int_0^T|u_s^{\pi,m}|^2ds\right].
\]
The finite-energy assumption implies $\tau_m\uparrow T$ and
$Y^{\pi,m}\to Y^\pi$ uniformly, $\mathbb Q^\pi$-a.s.  Lower semicontinuity of
relative entropy and monotone convergence therefore give
\[
 H(Q^\pi\mid P^\sigma)
 \le\frac1{2\sigma^2}\mathbb E_{\mathbb Q^\pi}
 \left[\int_0^T|u_s^\pi|^2ds\right],
\]
which proves part~\textup{(i)}.  For part~\textup{(ii)}, apply the finite-entropy construction of
\cite[Lemma~5.2]{nak:2024sb}, after the time-horizon and diffusion rescaling.
Its proof uses the density martingale, the generalized Girsanov--Maruyama theorem,
and localization to construct the drift and establish the energy identity.
\qed
\end{proof}

The constrained value \eqref{eq:4.Hstar} is defined over the weak class $\Pi_{\mu_T}$,
whereas the penalized problems are posed over feedback controls.
The following lemma shows that the two infima coincide, so that no approximation hypothesis is needed:
a finite constrained value already produces feedback controls that are exactly feasible and nearly optimal.

\begin{lem}[Feedback near-minimizers]\label{lem:feedback}
Assume $H^*_c<\infty$. Then for every $\eta>0$ there exists $u^\eta\in\mathcal{A}$ such that
\[
 \mathcal{L}(X^{u^\eta}_T)=\mu_T,\qquad
 \tfrac{1}{2}\mathcal{C}(u^\eta)+c\,\mathcal{R}_T(u^\eta)\;\le\;H^*_c+\eta .
\]
In particular $H^*_c=\inf\{\frac{1}{2}\mathcal{C}(u)+c\mathcal{R}_T(u):u\in\mathcal{A},\ \mathcal{L}(X^u_T)=\mu_T\}$.
\end{lem}

\begin{proof}
Let $\pi\in\Pi_{\mu_T}$ satisfy $J_c(\pi)\le H^*_c+\eta$;
such a $\pi$ exists by \eqref{eq:4.Hstar}, and $\EE_{\mathbb{Q}^\pi}\int_0^T|u^\pi_t|^2dt<\infty$.
Since $Y^\pi$ is an It\^o process with constant diffusion coefficient $\sigma$ and drift $u^\pi$,
the mimicking theorem of Brunick and Shreve \cite[Corollary~3.7]{bru-shr:2013}
provides a weak solution $X$ of \eqref{eq:4.0} with the feedback drift
\[
 u^\eta(t,x):=\EE_{\mathbb{Q}^\pi}\bigl[u^\pi_t\,\big|\,Y_t^\pi=x\bigr]
\]
such that $\mathcal{L}(X_t)=\mathbb{Q}^\pi\circ(Y_t^\pi)^{-1}$ for every $t\in[0,T]$.
By Jensen's inequality,
\[
 \mathcal{C}(u^\eta)
 =\EE_{Q^{u^\eta}}\int_0^T|\EE_{\mathbb{Q}^\pi}[u^\pi_t|Y_t^\pi]|^2dt
 \le\EE_{\mathbb{Q}^\pi}\int_0^T|u^\pi_t|^2dt<\infty,
\]
so $u^\eta\in\mathcal{A}$.
The time marginals are unchanged, hence
\[
 \mathcal{L}(X_T)=\mu_T,\qquad
 \mathcal{R}_T(u^\eta)=\int_0^T\mathcal{R}[\mathbb{Q}^\pi\circ(Y_t^\pi)^{-1}]dt,
\]
and the latter is the running cost of $\pi$.
Therefore $\frac{1}{2}\mathcal{C}(u^\eta)+c\mathcal{R}_T(u^\eta)\le J_c(\pi)\le H^*_c+\eta$.
The last assertion follows since every feedback control that is feasible induces an element of $\Pi_{\mu_T}$.
\qed
\end{proof}

We first state the population convergence result.  Its hypothesis is satisfied, in particular,
by any $\varepsilon_n$-minimizer of $H^*_{\lambda_n,c}$, because
Lemma~\ref{lem:feedback} gives $H^*_{\lambda,c}\le H^*_c$.
\begin{thm}[Convergence for the potential MFG]\label{thm:convergence}
Suppose that $\mu_0$ and $\mu_T$ have finite second moments and that $(A)$ holds
for the kernels $K$ and $W$.
Fix $c\ge0$, assume $H^*_c<\infty$, and let $\lambda_n\to\infty$ and
$\varepsilon_n\to0$.  If $u_n\in\mathcal A$ satisfies
\begin{equation}\label{eq:4.pop_opt}
 \mathcal J_{\lambda_n,c}(u_n)\le H^*_c+\varepsilon_n ,
\end{equation}
then
\begin{enumerate}[label=\textup{(\roman*)}]
\item $\displaystyle \sqrt{\lambda_n}\,
 \gamma_K(\mathcal L(X_T^{u_n}),\mu_T)\longrightarrow0$;
\item $\displaystyle \tfrac12\mathcal C(u_n)+c\mathcal R_T(u_n)\longrightarrow H^*_c$ and
$\mathcal J_{\lambda_n,c}(u_n)\longrightarrow H^*_c$.
\end{enumerate}
\end{thm}
\begin{proof}
See Section~\ref{sec:6.2}.
\end{proof}

Theorem~\ref{thm:convergence} is a population statement and does not assert
that the neural-network/SGD output of Algorithm~\ref{alg:sbp} satisfies
\eqref{eq:4.pop_opt}.  If $(u_n)$ is random (for example, data-dependent), the
same result applies pathwise on any event of probability one on which
\eqref{eq:4.pop_opt} holds; Corollary~\ref{cor:empirical_evaluation} then uses
new evaluation samples conditionally independent of these controls.

\begin{cor}[Schr\"{o}dinger-bridge specialization]\label{cor:sbp_special_case}
Suppose $c=0$, $T=1$, and $(A)$ holds for $K$.  If $H^*<\infty$ and
$\mathcal J_{\lambda_n,0}(u_n)\le H^*+\varepsilon_n$ with
$\lambda_n\to\infty$ and $\varepsilon_n\to0$, then the conclusions of
Theorem~{\rm \ref{thm:convergence}} hold with $H^*_c$ replaced by the Schr\"{o}dinger-bridge value
\[
 H^* := H^*_c\bigr|_{c=0,\,T=1}
 = \inf_{\pi\in\Pi_{\mu_1}}\frac{1}{2}\EE_{\mathbb{Q}^\pi}\!\left[\int_0^1|u^\pi_t|^2\,dt\right].
\]
\end{cor}

We next evaluate the same controls on fresh, independent samples.  Write
\[
\hat{\mathcal C}_{N_n}(u_n):=
\frac1{N_n}\sum_{i=1}^{N_n}\int_0^T|u_n(t,X_t^{(n,i)})|^2dt
\]

\begin{cor}[Independent sample-level evaluation]\label{cor:empirical_evaluation}
Under the hypotheses of Theorem~{\rm \ref{thm:convergence}}, let $M_n,N_n\to\infty$ satisfy
\begin{equation}\label{eq:4.cond}
 \lambda_n\sqrt{\frac{\log n}{M_n}}\to0,\qquad
 \lambda_n\sqrt{\frac{\log n}{N_n}}\to0.
\end{equation}
Assume that there exists a deterministic constant $C<\infty$ such that
$\sup_{n\ge1}\|u_n\|_\infty\le C$ almost surely.
At stage $n$, compute \eqref{eq:4.2} using, conditionally on $u_n$, four mutually
independent families satisfying
\[
\begin{gathered}
 X^{(n,i)}\stackrel{\mathrm{i.i.d.}}{\sim}Q^{u_n}\quad(1\le i\le N_n),
 \qquad Y_j^{(n)}\stackrel{\mathrm{i.i.d.}}{\sim}\mu_T\quad(1\le j\le N_n),\\
 Z_r^{(n)}\stackrel{\mathrm{i.i.d.}}{\sim}\rho_K/\Phi(0),
 \qquad \widetilde Z_r^{(n)}\stackrel{\mathrm{i.i.d.}}{\sim}\rho_W/\Psi(0)
 \quad(1\le r\le M_n).
\end{gathered}
\]
Then
\[
 \hat{\mathcal J}_{\lambda_n,c,M_n,N_n}(u_n)\longrightarrow H^*_c
 \qquad\text{a.s.}
\]
\end{cor}
\begin{proof}
See Section~\ref{sec:6.2}.
\end{proof}

All sample-level statements of this section concern the continuous-time evaluation
functional \eqref{eq:4.2}, in which the controlled paths solve \eqref{eq:4.0}
exactly and the time integrals are exact.
Algorithm~\ref{alg:sbp} replaces them by an Euler--Maruyama scheme and a Riemann sum on a time grid;
the resulting discretization error is not analysed here and is treated as part of the implementation.

\begin{prop}[Rate of convergence, SBP case]\label{prop:rate}
Under the hypotheses of Corollary~{\rm \ref{cor:sbp_special_case}},
\[
 \gamma_K\bigl(\mathcal L(X_1^{u_n}),\mu_1\bigr)^2
 \le \frac{H^*+\varepsilon_n}{\lambda_n}.
\]
If, conditionally on $u_n$, the three families
\[
\begin{gathered}
 X_1^{(n,i)}\stackrel{\mathrm{i.i.d.}}{\sim}\mathcal L(X_1^{u_n})\quad(1\le i\le N_n),
 \qquad Y_j^{(n)}\stackrel{\mathrm{i.i.d.}}{\sim}\mu_1\quad(1\le j\le N_n),\\
 Z_r^{(n)}\stackrel{\mathrm{i.i.d.}}{\sim}\rho_K/\Phi(0)\quad(1\le r\le M_n)
\end{gathered}
\]
are mutually independent, then with $C^*:=12\Phi(0)$,
for all sufficiently large $n$, almost surely,
\begin{equation}\label{eq:4.rate_emp}
 \hat{\gamma}_{M_n,N_n}^2\bigl(\mathcal{L}(X_1^{u_n}),\mu_1\bigr)
 \;\le\; \frac{H^* + \varepsilon_n}{\lambda_n}
 \;+\; C^*\!\left(\sqrt{\frac{\log n}{M_n}} + \sqrt{\frac{\log n}{N_n}}\right),
\end{equation}
provided $M_n/\log n,N_n/\log n\to\infty$.
For $\lambda_n=n^a$, $M_n=N_n=n^b$, $\varepsilon_n=O(1)$ and $b>2a$,
\[
 \gamma_K\bigl(\mathcal{L}(X_1^{u_n}),\mu_1\bigr)=O(n^{-a/2}).
\]
\end{prop}

\begin{proof}
See Section~\ref{sec:6.3rate}.
\end{proof}

\begin{prop}[Rate of convergence for the potential MFG]\label{prop:rate_mfg}
Under the hypotheses of Theorem~{\rm \ref{thm:convergence}},
\begin{equation}\label{eq:4.rate_mfg}
 \gamma_K\bigl(\mathcal L(X_T^{u_n}),\mu_T\bigr)^2
 \le \frac{H^*_c+\varepsilon_n}{\lambda_n}.
\end{equation}
The observable estimator obeys the analogue of \eqref{eq:4.rate_emp} with $H^*$ replaced by $H^*_c$
on an independent evaluation batch, provided the conditional independence assumptions of
Proposition~{\rm \ref{prop:rate}} hold and $M_n/\log n,N_n/\log n\to\infty$.
\end{prop}
\begin{proof}
See Section~\ref{sec:6.3rate}.
\end{proof}

The preceding polynomial schedule is only a sufficient asymptotic prescription.
If $\lambda_n=n^a$ and $M_n=N_n=n^b$, the almost-sure evaluation bound requires the
strict inequality $b>2a$; the apparently balanced choice $b=2a$ does not absorb the
$\sqrt{\log n}$ factor.  In finite computations, we instead recommend increasing
$\lambda$ over a moderate grid until the held-out terminal MMD ceases to improve
materially, while monitoring the control energy and optimization stability.
The feature count $M$ should then be enlarged until the random-feature estimate is
stable across frequency draws, and the held-out batch size $N$ until its sampling
variability is acceptably small.  Thus the fixed values used in Section~\ref{sec:5}
are computational tuning choices, not an implementation of the asymptotic schedule.

In the independent evaluation result, the limit $M\to\infty$ corresponds to replacing the random Fourier
$U$-statistic by the kernel $U$-statistic estimator $\bar{\gamma}_K^2(D)$
(cf.~\eqref{eq:empirical_mmd} and the proof of Theorem~\ref{thm:rfu_unbiased}(iii) in Section~\ref{sec:6.1}),
evaluated on the controlled-path terminal samples $X_1^{(1)},\ldots,X_1^{(N)}$ and target i.i.d.\ samples $Y^{(1)},\ldots,Y^{(N)}\sim\mu_T$:
\begin{equation}\label{eq:4.kernelUstat}
 \bar{\gamma}_K^2(D)
 = \frac{1}{N(N-1)}\sum_{i\neq j}K(X_1^{(i)},X_1^{(j)})
   - \frac{2}{N^2}\sum_{i,j}K(X_1^{(i)},Y^{(j)})
   + \frac{1}{N(N-1)}\sum_{i\neq j}K(Y^{(i)},Y^{(j)}).
\end{equation}
This is the population MMD${}^2$ estimator used in \cite{nak:2024sb}.
For a fixed control $u$, the corresponding kernel-$U$-statistic evaluation functional is
\begin{equation}\label{eq:4.kernel_pen}
 \hat{\mathcal J}_{\lambda,N}^{\mathrm{ker}}(u)
 :=
 \frac{1}{2N}\sum_{i=1}^N\!\int_0^1|u(t,X_t^{(i)})|^2 dt
 + \lambda\,\bar{\gamma}_K^2(D).
\end{equation}

The kernel-$U$-statistic objective is the $M=\infty$ counterpart of
\eqref{eq:4.2}.  For each fixed control and evaluation batch,
\eqref{eq:4.kernel_pen} is the $M=\infty$ limit of the random Fourier
evaluation functional \eqref{eq:4.2}
in the sense that $\hat{\gamma}_{M,N}^2\to \bar{\gamma}_K^2(D)$ a.s.\ as $M\to\infty$ for fixed $N$
(by the conditional tower decomposition \eqref{eq:6.tower}; see Remark~\ref{rem:6.var}),
so the next corollary is the corresponding independent evaluation result without a random-feature condition:

\begin{cor}\label{cor:kernel_convergence}
Suppose that $c=0$ and $T=1$, that $\mu_0$ and $\mu_1$ have finite second moments,
that $(A)$ holds for $K$, and that $H^*<\infty$.
Let $\{\lambda_n\}$, $\{N_n\}$, and $\{\varepsilon_n\}$ be deterministic sequences satisfying
\[
 \lambda_n\to\infty,\quad \varepsilon_n\to 0,\quad \text{and}\quad \lambda_n\sqrt{\log n/N_n}\to 0.
\]
Let $u_n$ satisfy $\mathcal J_{\lambda_n,0}(u_n)\le H^*+\varepsilon_n$.
Assume that $\sup_n\|u_n\|_\infty<\infty$ almost surely, with a deterministic bound,
and evaluate it on a conditionally i.i.d.\ batch independent of $u_n$, with
$\hat{\mathcal J}^{\rm ker}_{\lambda_n,N_n}(u_n)$ denoting the objective inside
\eqref{eq:4.kernel_pen}.  Then
$\hat{\mathcal J}^{\rm ker}_{\lambda_n,N_n}(u_n)\to H^*$ almost surely.
\end{cor}

\begin{proof}
See Section~\ref{sec:6.3rate}.
\end{proof}

\begin{prop}[General running costs]\label{prop:convergence_general}
Theorem~{\rm \ref{thm:convergence}} remains valid when the kernel interaction is replaced by any
non-negative weakly continuous $\mathcal R:\cP(\RR^d)\to[0,\infty)$.
For the sample-level statement, retain the hypotheses of
Corollary~{\rm \ref{cor:empirical_evaluation}}; in particular, assume
$\sup_{n\ge1}\|u_n\|_\infty\le C$ almost surely for some deterministic $C<\infty$.

Suppose, for some fixed $m\ge1$ and
$B<\infty$, that the time-integrated estimator has the conditionally unbiased
$U$-statistic representation
\[
 \hat{\mathcal R}_T(u_n)
 :=\int_0^T\hat{\mathcal R}_{N_n}[\hat\nu_t^{N_n}]dt
 =\binom{N_n}{m}^{-1}
 \sum_{1\le i_1<\cdots<i_m\le N_n}
 h_n(X^{(n,i_1)},\ldots,X^{(n,i_m)}),
\]
where, conditionally on $u_n$, the evaluation paths are i.i.d., $h_n$ is a
symmetric $\sigma(u_n)$-measurable kernel, $\|h_n\|_\infty\le B$ almost surely, and
\[
 \mathbb E[h_n(X^{(n,1)},\ldots,X^{(n,m)})\mid u_n]
 =\mathcal R_T(u_n).
\]
Then the conclusion of Corollary~{\rm \ref{cor:empirical_evaluation}} remains valid.
\end{prop}

\subsection{Algorithm}\label{sec:4.3alg}

The population objective \eqref{eq:4.1mfg} is approached computationally by
parametrizing the drift $u_\theta$ of the controlled SDE \eqref{eq:4.0} as a neural network
and using the freshly resampled loss \eqref{eq:4.2} in stochastic gradient descent;
see Algorithm~\ref{alg:sbp}.
At each iteration the empirical objective is evaluated by combining
the unbiased $O(dNM)$ estimators of $\gamma_K^2$ (Theorem~\ref{thm:rfu_unbiased})
and of $\mathcal{R}[\nu_t]$ (Theorem~\ref{thm:rfu_interaction}).
The Schr{\"o}dinger-bridge specialization $c=0$ recovers the algorithm of \cite{nak:2024sb}
with the $O(dN^2)$ kernel $U$-statistic replaced by the $O(dNM)$ RF $U$-statistic.

\begin{algorithm}
\caption{Potential MFG with kernel costs via RF $U$-statistic estimators}
\label{alg:sbp}
\begin{algorithmic}[1]
\Input Number $n$ of iterations, batch size $N$, random-feature count $M$,
terminal penalty $\lambda>0$, congestion weight $c\ge 0$,
terminal kernel bandwidth $\alpha_K$, congestion kernel bandwidth $\alpha_W$,
diffusion coefficient $\sigma>0$, time horizon $T>0$,
time grid $0=t_0<t_1<\cdots<t_q=T$ with $q$ time intervals and $q+1$ grid points.
\Output Neural-network drift $u_\theta$
\State Initialize $\theta$
\For{$\ell=1,2,\ldots,n$}
      \State $\{Y_j\}_{j=1}^N \gets$ i.i.d.\ samples from $\mu_T$
      \State $\{Z_r\}_{r=1}^M \gets$ i.i.d.\ samples from $\rho_K/\Phi(0)$, $\{\tilde Z_r\}_{r=1}^M \gets$ i.i.d.\ from $\rho_W/\Psi(0)$
      \State Simulate $N$ independent paths $\{X_t^{(i)}\}_{t\in[0,T]}$, $i=1,\ldots,N$, via Euler--Maruyama
      \State $\hat{\mathcal{C}}(\theta) \gets \frac{1}{2N}\sum_{i=1}^N \int_0^T|u_\theta(t,X_t^{(i)})|^2dt$
      \Comment{empirical energy, as in \eqref{eq:4.2}}
      \State $\hat{\gamma}_{M,N}^2 \gets$ \eqref{eq:3.4} using $\{X_T^{(i)}\}_{i=1}^N$, $\{Y_j\}$, $\{Z_r\}$
      \Comment{terminal MMD$^2$ penalty}
      \If{$c>0$}
        \State $\hat{\mathcal{R}}_{\mathrm{tot}}(\theta) \gets \sum_{k=1}^{q}\hat{\mathcal{R}}_{M,N}[\hat\nu^N_{t_k}]\,(t_k - t_{k-1})$
        \Comment{time-integrated empirical interaction via \eqref{eq:3.5}}
      \Else
        \State $\hat{\mathcal{R}}_{\mathrm{tot}}(\theta) \gets 0$
      \EndIf
      \State $F(\theta) \gets \frac{1}{\lambda}\hat{\mathcal{C}}(\theta) + \frac{c}{\lambda}\hat{\mathcal{R}}_{\mathrm{tot}}(\theta) + \hat{\gamma}_{M,N}^2$
      \Comment{$\lambda F$ is the objective of \eqref{eq:4.2}}
      \State Take gradient step on $\nabla_\theta F(\theta)$
\EndFor
\end{algorithmic}
\end{algorithm}

Let $P_\theta$ denote, up to a constant factor, the computational cost per state--time input
of a forward pass through $u_\theta$ together with the corresponding backward pass.
Simulating the $N$ paths and backpropagating through $q$ time intervals therefore costs $O(qNP_\theta)$.
The terminal RF penalty costs $O(dNM)$, while evaluating the RF interaction at all $q$
time levels costs $O(qdNM)$.  Thus the full $c>0$ iteration has cost
$O\bigl(qN(P_\theta+dM)\bigr)$; for the SBP specialization $c=0$, it is
$O(qNP_\theta+dNM)$.  These counts include the $d$-dimensional feature projections
that are suppressed when $O(NM)$ is used with fixed $d$.

\begin{rem}\label{rem:3.interaction}
The interaction estimator \eqref{eq:3.5} is computed on a \emph{single} batch of $N$ controlled-path samples,
in contrast to the terminal estimator \eqref{eq:3.4} which requires both controlled-path and target samples.
In Algorithm~\ref{alg:sbp}, the same batch $\{X^{(i)}\}_{i=1}^N$
is used to evaluate $\hat{\mathcal{R}}_{M,N}[\hat\nu^N_{t_k}]$ at each time step $t_k$ along the trajectory
and to evaluate $\hat\gamma_{M,N}^2(\hat\nu^N_T,\mu_T)$ at the terminal time;
the target samples $\{Y_j\}_{j=1}^N$ from $\mu_T$ and the two independent collections of random frequencies
$\{Z_r\}$ for $K$ and $\{\tilde Z_r\}$ for $W$ enter only through the terminal and interaction penalties, respectively.
\end{rem}

\begin{rem}\label{rem:3.3}
At each iteration, the random frequencies $Z_1,\ldots,Z_M$, the path noise driving the $N$ Euler--Maruyama trajectories,
and the target samples $Y_1,\ldots,Y_N$ are drawn independently and afresh.
This ``re-sampling'' strategy makes the kernel-cost terms conditionally unbiased
for the corresponding population costs (Theorems~{\rm \ref{thm:rfu_unbiased}}(i)
and~{\rm \ref{thm:rfu_interaction}}(i)).  Subject to the usual interchange of
differentiation and expectation, it therefore gives the associated unbiased stochastic gradients.
It also provides the sample-independence used in the convergence analysis above
when the final control is assessed on a held-out batch
(Corollary~{\rm \ref{cor:empirical_evaluation}}).
The $U$-statistic structure of $\hat{\gamma}_{M,N}^2$ does not require splitting the $N$ paths into sub-batches:
the diagonal-removing decomposition \eqref{eq:3.2}--\eqref{eq:3.3} already eliminates the biasing self-terms,
so all $N$ paths participate jointly as the $X$-sample of the estimator.
\end{rem}

\begin{rem}\label{rem:3.4}
The same estimator applies to the Monge optimal transport problem studied in \cite{sai-nak:2025},
as well as to its multi-marginal extension in \cite{nak-sai:2026}.
In the Monge setting, the SDE is replaced by a deterministic transport map $T_\theta$,
and the objective becomes $\frac{1}{\lambda N}\sum_i c(X_i,T_\theta(X_i)) + \hat{\gamma}_{M,N}^2$
with $\hat{\gamma}_{M,N}^2$ computed from $\{T_\theta(X_i)\}$ and target samples $\{Y_j\}$.
\end{rem}

\begin{rem}[Interaction-dependent dynamics]\label{rem:interaction_drift}
Algorithm~\ref{alg:sbp} and the convergence results above assume that the drift is
$u(t,x)$ and that the law enters through the cost.  Replacing it by
$b(t,x,\nu_t)+u(t,x)$ is not a formal change to the present proof: it requires
well-posedness and stability of the controlled McKean--Vlasov equation, together
with propagation-of-chaos or concentration estimates for the interacting particle
system used in training and evaluation.  For convolutional $b$, random Fourier
features may reduce the cost of evaluating the empirical convolution, but they do
not supply these probabilistic ingredients.  We therefore leave this extension to
future work.
\end{rem}

\section{Numerical experiments}\label{sec:5}

All numerical experiments are implemented in PyTorch.
The full source code, together with scripts that reproduce every figure and table of this section,
is available at \url{https://github.com/yumiharu-nakano/kernelMFG-RFU}.
The Gaussian kernel $K(x,y)=e^{-\alpha|x-y|^2}$ is used throughout for the terminal MMD penalty,
and the Gaussian congestion kernel $W(x,y)=e^{-\alpha_W|x-y|^2}$ is used in the
standalone and end-to-end interaction tests.  The EV extension of Section~\ref{sec:5ev}
instead uses the applied charging power in a control-dependent aggregate-demand cost.
The bandwidth $\alpha$ is specified in each subsection;
for the high-dimensional Schr{\"o}dinger-bridge experiments we adopt the standard scaling $\alpha=1/d$,
while for the low-dimensional ($d=2$) estimator-property and bimodal experiments we use $\alpha=1$.
Unless stated otherwise, the drift network is trained with Adam at learning rate $10^{-3}$,
without learning-rate schedule, gradient clipping or weight decay,
and each iteration (one epoch in the tables below) draws a fresh batch of paths, target samples and random frequencies,
as in Algorithm~\ref{alg:sbp}.

We organize the experiments as follows.
Sections~\ref{sec:est_prop}--\ref{sec:5.5comp} treat the Schr{\"o}dinger-bridge specialization
($c=0$ in the potential MFG \eqref{eq:4.1mfg}, Corollary~\ref{cor:sbp_special_case})
as a controlled warm-up of the framework that
checks the estimator properties and provides numerical diagnostics consistent with the convergence theory
and demonstrates the sample-linear scaling for fixed $d$ and $M$.
Section~\ref{sec:5ev} then presents a physically consistent EV charging extension
with a control-dependent aggregate-demand term and explicitly separates it from
the scope of the convergence theorem.

\subsection{Estimator properties}\label{sec:est_prop}

We first verify the two key properties of the proposed RF $U$-statistic estimator
established in Theorem~\ref{thm:rfu_unbiased}: unbiasedness and the variance bound
$\mathrm{Var}(\hat{\gamma}_{M,N}^2) = O(1/M)+O(1/N)$.

\subsubsection{Bias comparison}\label{sec:5.1}

Before presenting the SBP experiments, we verify the unbiasedness of the proposed estimator.
We compare three estimators of $\gamma_K(\mu,\nu)^2$:
(i) the kernel $U$-statistic \eqref{eq:empirical_mmd} with $O(dN^2)$ complexity,
(ii) the standard RFF $V$-statistic (Rahimi--Recht \cite{rah-rec:2007}) with $O(dNM)$ complexity, and
(iii) the proposed RF $U$-statistic \eqref{eq:3.4} with $O(dNM)$ complexity.
For $d=2$, $N=200$, $M=200$, and $\alpha=1$,
we compute each estimator $2000$ times with $\mu=\nu=\mathcal{N}(0,I_2)$ so that $\gamma_K^2=0$.

\begin{table}[htbp]
\centering
\caption{Bias comparison with $\mu=\nu=\mathcal{N}(0,I_2)$ (true $\gamma_K^2=0$), $2000$ trials.}
\label{tab:bias_same}
\begin{tabular}{lcc}
\toprule
Estimator & Mean & Std \\
\midrule
Kernel $U$-stat \eqref{eq:empirical_mmd} & $-3\times 10^{-5}$ & $3.3\times 10^{-3}$ \\
RFF $V$-stat (Rahimi--Recht) & $+8.0\times 10^{-3}$ & $3.3\times 10^{-3}$ \\
RF $U$-stat (proposed) \eqref{eq:3.4} & $-2\times 10^{-5}$ & $3.3\times 10^{-3}$ \\
\bottomrule
\end{tabular}
\end{table}

The kernel $U$-statistic and the proposed RF $U$-statistic both have sample means close to zero,
confirming their unbiasedness.
In contrast, the RFF $V$-statistic exhibits a positive bias of approximately $+8.0\times 10^{-3}$,
consistent with the $O(\Phi(0)/N)=O(1/200)$ prediction of Remark \ref{rem:3.2}.
At the population level this bias produces a non-vanishing penalty floor even when $\mathcal{L}(X_1)=\mu_1$ exactly,
which complicates the limit-theorem analysis as $\lambda\to\infty$
(cf.\ Remark~\ref{rem:3.2}).

\subsubsection{Variance scaling}\label{sec:5.1var}

We numerically verify Theorem~\ref{thm:rfu_unbiased}(iii):
\[
\mathrm{Var}(\hat{\gamma}_{M,N}^2)=O(1/M)+O(1/N),
\]
in which the $O(1/M)$ term arises from the random Fourier averaging
and the $O(1/N)$ term is the irreducible kernel $U$-statistic variance
identified in the decomposition \eqref{eq:6.tower}.

We take $\mu=\mathcal{N}(0,I_d)$ and $\nu=\mathcal{N}(m,I_d)$
with $m=(1,0,\ldots,0)^{\mathsf{T}}\in\RR^{10}$ (so the alternative case $\gamma_K^2>0$),
$d=10$, and $\alpha=1/d$.
For each $(M,N)$ in the grid $M\in\{50,100,200,500,1000,2000,5000\}$, $N\in\{50,100,200,500,1000\}$,
we compute the estimator $\hat{\gamma}_{M,N}^2$ over $T=2000$ independent trials
and report the sample variance.

\begin{table}[htbp]
\centering
\caption{Sample variance of $\hat{\gamma}_{M,N}^2$ ($\times 10^{-4}$) on a grid of $(M,N)$.
$d=10$, $\alpha=1/d$, $\mu=\mathcal{N}(0,I)$, $\nu=\mathcal{N}(m,I)$ with $m_1=1$,
over $2000$ trials per cell.
The sample mean is constant at $0.0256\pm 0.0003$ across all cells (true $\gamma_K^2\approx 0.0256$),
confirming unbiasedness.}
\label{tab:variance_scaling}
\begin{tabular}{c|ccccc}
\toprule
$M\backslash N$ & $50$ & $100$ & $200$ & $500$ & $1000$ \\
\midrule
$50$   & $2.21$ & $1.20$ & $0.69$ & $0.38$ & $0.29$ \\
$100$  & $1.88$ & $0.89$ & $0.48$ & $0.24$ & $0.19$ \\
$200$  & $1.67$ & $0.73$ & $0.40$ & $0.19$ & $0.12$ \\
$500$  & $1.48$ & $0.67$ & $0.34$ & $0.14$ & $0.08$ \\
$1000$ & $1.53$ & $0.69$ & $0.32$ & $0.13$ & $0.07$ \\
$2000$ & $1.55$ & $0.66$ & $0.31$ & $0.13$ & $0.07$ \\
$5000$ & $1.49$ & $0.64$ & $0.32$ & $0.12$ & $0.06$ \\
\bottomrule
\end{tabular}
\end{table}

A non-negative least-squares fit of the model $\mathrm{Var}=c_1/M + c_2/N$
to the $35$ cells of Table~\ref{tab:variance_scaling} yields
\[
 \mathrm{Var}(\hat{\gamma}_{M,N}^2) \approx \frac{0.0017}{M} + \frac{0.0077}{N},
\]
with both coefficients positive and the model explaining the data uniformly across the grid.
Two qualitative checks support the predicted scaling:
(i) at the largest $M=5000$ (so the random-feature contribution $c_1/M$ is negligible),
linear regression of $\log\mathrm{Var}$ on $\log(1/N)$ gives slope $1.06$ (theoretical: $1$);
(ii) at the largest $N=1000$, the variance plateaus near $c_2/N=7.7\times 10^{-6}$ as $M$ grows,
matching the floor predicted by the kernel $U$-statistic component.

The decomposition is also visible in Table~\ref{tab:variance_scaling}:
each row stabilizes at the irreducible $O(1/N)$ floor as $M\to\infty$,
in agreement with $\hat{\gamma}_{M,N}^2\to\bar{\gamma}_K^2(D)$ from \eqref{eq:6.tower}.
The sample mean across all $7\times 5\times 2000=7\times 10^4$ trials is $0.0256$
with cell-to-cell standard deviation below $3\times 10^{-4}$,
providing additional confirmation of the unbiasedness in the alternative case.

\subsubsection{Interaction estimator: bias and variance}\label{sec:5.1.4interaction}

We now verify Theorem~\ref{thm:rfu_interaction}, the analogous properties for the
\emph{interaction} estimator $\hat{\mathcal{R}}_{M,N}[\hat\nu^N]$ of \eqref{eq:3.5}.
We take $\nu=\mathcal{N}(0, I_d)$ in dimension $d=2$ with the Gaussian congestion kernel
$W(x,y)=\exp(-\alpha|x-y|^2)$ at $\alpha=1.0$,
for which $\mathcal{R}[\nu]=\frac{1}{2}(1+4\alpha)^{-d/2}$ has the closed form $0.1$.
This provides an analytic target against which the empirical estimator is compared,
in parallel with Sections~\ref{sec:5.1}--\ref{sec:5.1var} for the terminal MMD${}^2$ estimator.

\paragraph{Bias check.}
With $N=200$ samples and $M=500$ random frequencies,
averaged over $2{,}000$ independent trials,
the proposed unbiased estimator $\hat{\mathcal{R}}_{M,N}$ gives mean $0.10027\pm 1.9\times 10^{-4}$
(standard error of the mean across trials),
matching the analytic value $0.1$ within Monte Carlo precision (bias $+2.7\times 10^{-4}$,
within the standard error).
For comparison, the corresponding $V$-statistic estimator
\[
 \hat{V}_{M,N}^{\mathcal{R}}[\hat\nu^N]
 := \frac{\Psi(0)}{2MN^2}\sum_{r=1}^M\bigl(S_c^X(\tilde Z_r)^2+S_s^X(\tilde Z_r)^2\bigr)
\]
gives mean $0.10179\pm 1.9\times 10^{-4}$
with bias $+1.8\times 10^{-3}$, very close to the theoretical prediction
$\frac{\Psi(0)}{2N}=\frac{1}{400}=2.5\times 10^{-3}$ for the Gaussian kernel.
This directly verifies Theorem~\ref{thm:rfu_interaction} (i) and the parallel of Remark~\ref{rem:3.2}
for the interaction estimator.

\paragraph{Variance decomposition.}
On a grid of $5$ sample sizes $N\in\{50, 100, 200, 500, 1000\}$
and $6$ random-feature counts $M\in\{20, 50, 100, 500, 10^3, 5\!\cdot\!10^3\}$,
each cell averaged over $500$ trials,
we compute the empirical variance $\mathrm{Var}(\hat{\mathcal{R}}_{M,N})$
and fit the two-parameter model $\mathrm{Var}\approx c_1/M + c_2/N$
by non-negative least squares.
The best fit is
\[
 \mathrm{Var}(\hat{\mathcal{R}}_{M,N}) \;\approx\; \frac{0.0177}{M} + \frac{0.0085}{N},
\]
with both coefficients positive and the model explaining $R^2=0.9963$ of the variation
across the $30$-cell grid.
This confirms the variance decomposition $O(1/M)+O(1/N)$ of Theorem~\ref{thm:rfu_interaction} (iii),
in the same form (and with comparable coefficient magnitudes) as the terminal MMD${}^2$ estimator
of Theorem~\ref{thm:rfu_unbiased}(iii) reported in Section~\ref{sec:5.1var}.
The parallel structure of the two estimators is therefore reflected at the empirical level
as well as at the analytic level,
justifying the use of the same $(M,N)$-coupling in the convergence analysis of Section~\ref{sec:4.2conv}
for both the terminal MMD penalty and the kernel congestion cost.

\paragraph{End-to-end interaction test.}
We also insert the same estimator into the complete controlled-path objective, rather than testing it only in isolation.
For $d=2$, $\mu_0=\delta_0$, $\mu_1=\mathcal N((2,0)^{\mathsf T},I_2)$, and $\sigma=0.5$,
we train with $c=0$ or $c=5$ and the translation-invariant Gaussian interaction
$W(x,y)=\exp(-\alpha_W|x-y|^2)$; the terminal MMD kernel has $\alpha_K=0.5$.
Table~\ref{tab:interaction_end_to_end} reports three seeds with $N=128$, $M_K=200$, $M_W=64$,
$\lambda^{-1}=0.02$, and $2000$ epochs.  The interaction raises the total variance at $t=1/2$,
as expected for this repulsive congestion potential, while leaving the terminal MMD and control energy nearly unchanged.
The effect is consistent across the limited bandwidth range $\alpha_W\in\{0.5,1,2\}$;
this is a matched end-to-end check, not a comprehensive hyperparameter study.

\begin{table}[htbp]
\centering
\caption{End-to-end Gaussian interaction test (mean $\pm$ standard deviation over $3$ seeds).
The control energy is $\frac12\EE\int_0^1|u|^2dt$, and the midpoint total variance is
$\sum_{k=1}^2\operatorname{Var}(X_{1/2,k})$.}
\label{tab:interaction_end_to_end}
\begin{tabular}{ccccc}
\toprule
$c$ & $\alpha_W$ & Terminal MMD${}^2$ & Midpoint total variance & Control energy \\
\midrule
$0$ & --    & $(9.99\pm6.61)\times10^{-3}$ & $0.558\pm0.014$ & $1.75\pm0.18$ \\
$5$ & $0.5$ & $(9.37\pm6.13)\times10^{-3}$ & $0.679\pm0.030$ & $1.77\pm0.17$ \\
$5$ & $1.0$ & $(9.23\pm6.00)\times10^{-3}$ & $0.682\pm0.026$ & $1.77\pm0.17$ \\
$5$ & $2.0$ & $(9.29\pm6.04)\times10^{-3}$ & $0.668\pm0.023$ & $1.77\pm0.17$ \\
\bottomrule
\end{tabular}
\end{table}

\subsubsection{Effect of the penalty bias on SBP training}\label{sec:5.1.3vs}

The microbenchmarks above quantify the bias and variance of the estimator in isolation.
We now examine how the bias affects the practical Schr{\"o}dinger bridge training,
where the estimator is used inside an SGD loop with $\lambda\gg 1$ over thousands of iterations.
We compare the proposed unbiased RF $U$-statistic $\hat{\gamma}_{M,N}^2$
with the standard biased RFF $V$-statistic
$\hat{V}_{M,N}^2:= \Phi(0)|\frac{1}{N}\sum_i \phi(X_i)-\frac{1}{N}\sum_j\phi(Y_j)|^2$
of Rahimi--Recht~\cite{rah-rec:2007}.
We test two target families and several $(N,\lambda)$ regimes,
keeping all other hyperparameters at the values of Section~\ref{sec:5.2}--\ref{sec:5.3}.
Evaluation MMD${}^2$ is computed by the unbiased kernel $U$-statistic in all cases
so that the comparison is fair.

\paragraph{Unimodal target (Gaussian shift, $d=10$).}
We sweep the four configurations
$(N,\lambda^{-1})\in\{(80,2\times10^{-3}),(80,2\times10^{-5}),(20,2\times10^{-3}),(8,2\times10^{-3})\}$
and report mean $\pm$ standard deviation over $5$ seeds.
Table~\ref{tab:vstat_penalty} shows that the $V$-statistic penalty produces a trained drift
\emph{statistically indistinguishable} from the $U$-statistic penalty,
with differences well within the seed standard deviation across all four configurations.
The empirical $V$-statistic floor $\hat{V}(\mu_1,\mu_1)$
tracks the theoretical prediction $2\Phi(0)/N$ (predicted $0.025$, $0.10$ and $0.25$ at $N=80$, $20$ and $8$;
measured $0.020$, $0.085$ and $0.21$),
confirming that the bias is present at the predicted magnitude
but does not propagate into the optimization output.

\begin{table}[htbp]
\centering
\caption{Effect of the penalty bias on SBP training, unimodal Gaussian shift ($d=10$, $5$ seeds).
The trained drift is statistically indistinguishable between the two penalties across all $(N,\lambda)$ configurations,
even though the $V$-stat floor $\hat{V}(\mu_1,\mu_1)$ scales as $2\Phi(0)/N$,
with predicted value $0.25$ and measured value $0.21$ at $N=8$.}
\label{tab:vstat_penalty}
\begin{small}
\begin{tabular}{cccccc}
\toprule
$N$ & $\lambda^{-1}$ & Penalty & $\EE[X_1]_1$ & $|\EE[X_1]-m|^2$ & $\hat{V}(\mu_1,\mu_1)$ \\
\midrule
$80$ & $2\times10^{-3}$ & $U$-stat & $2.72\pm 0.14$ & $0.15\pm 0.07$ & $(2.0\pm 0.3)\times 10^{-2}$ \\
$80$ & $2\times10^{-3}$ & $V$-stat & $2.73\pm 0.14$ & $0.14\pm 0.07$ & $(2.0\pm 0.3)\times 10^{-2}$ \\
$80$ & $2\times10^{-5}$ & $U$-stat & $2.93\pm 0.14$ & $0.07\pm 0.02$ & $(2.0\pm 0.3)\times 10^{-2}$ \\
$80$ & $2\times10^{-5}$ & $V$-stat & $2.93\pm 0.14$ & $0.07\pm 0.01$ & $(2.0\pm 0.3)\times 10^{-2}$ \\
$20$ & $2\times10^{-3}$ & $U$-stat & $2.75\pm 0.11$ & $0.14\pm 0.04$ & $(8.5\pm 0.6)\times 10^{-2}$ \\
$20$ & $2\times10^{-3}$ & $V$-stat & $2.80\pm 0.13$ & $0.12\pm 0.03$ & $(8.5\pm 0.6)\times 10^{-2}$ \\
$\phantom{0}8$  & $2\times10^{-3}$ & $U$-stat & $2.49\pm 0.11$ & $0.45\pm 0.11$ & $(2.1\pm 0.2)\times 10^{-1}$ \\
$\phantom{0}8$  & $2\times10^{-3}$ & $V$-stat & $2.64\pm 0.07$ & $0.29\pm 0.09$ & $(2.1\pm 0.2)\times 10^{-1}$ \\
\bottomrule
\end{tabular}
\end{small}
\end{table}

\paragraph{Interpretation.}
The same comparison repeated on the bimodal target of Section~\ref{sec:5.2}
($N=64$, that is, half the batch size of Section~\ref{sec:5.2}, $\lambda^{-1}=2\times10^{-2}$, $30$ seeds)
is again statistically indistinguishable
($p=0.35$ in a one-sided Fisher exact test of mode-imbalance failure rates),
so the two penalties produce empirically indistinguishable trained drifts
across batch sizes $N\in[8,80]$ and penalty parameters $\lambda^{-1}\in[2\times10^{-5},2\times10^{-2}]$.
The reason is that the bias is approximately
$(1/N)\bigl[2\Phi(0)-\iint K(x,x')\nu_\theta(dx)\nu_\theta(dx')-\iint K(y,y')\mu_1(dy)\mu_1(dy')\bigr]$. 
Near the SBP optimum $\nu_\theta\approx\mu_1$ the two self-kernel expectations balance,
so the bias is approximately constant in $\theta$ and contributes only a small additional gradient.
We nevertheless adopt the unbiased $U$-statistic in the remainder of the paper for theoretical reasons. 
It eliminates the population-level penalty floor of order $\Phi(0)/N$,
removes the need for a bias-correction term in the convergence analysis,
and yields the variance decomposition $O(1/M)+O(1/N)$ of Theorem~\ref{thm:rfu_unbiased}(iii).

\subsection{Bimodal target ($d=2$)}\label{sec:5.2}

We solve the SBP with $\mu_0=\delta_0$ and a bimodal target
$\mu_1 = \frac{1}{2}\mathcal{N}((2,2)^{\mathsf{T}}, \frac{1}{4} I_2)
        + \frac{1}{2}\mathcal{N}((-2,-2)^{\mathsf{T}}, \frac{1}{4} I_2)$
in dimension $d=2$, with $\sigma=0.5$.
The drift $u_\theta(t,x)$ is parametrized by a network with hidden dimensions $(64,32)$ and LayerNorm.
We compare the proposed RF $U$-statistic ($M=200$) with the kernel estimator \eqref{eq:empirical_mmd},
using $N=128$, $\lambda^{-1}=0.02$, and $\alpha=1.0$ for $2000$ epochs under the same random seed.

Both methods successfully learn the bimodal structure of $\mu_1$.
The terminal distributions $X_1^{(\theta)}$ split into the two modes in both cases.
Averaged over $5$ independent runs, the RF $U$-statistic achieves a terminal MMD${}^2$ of
$(1.4\pm 1.0)\times 10^{-2}$, comparable to the kernel estimator's MMD${}^2 \approx 1.1\times 10^{-2}$,
confirming that the RF approximation does not degrade accuracy in the bimodal setting.
The training times are similar at $N=128$ ($10.6$s vs $10.9$s),
as the SDE integration dominates the per-iteration cost for small batch sizes.
However, the computational advantage of the RF method grows with $N$
(see Section \ref{sec:5.5comp}).

\subsection{High-dimensional Gaussian shift}\label{sec:5.3}

We test the method in higher dimensions with $\mu_0=\delta_0$ and
$\mu_1=\mathcal{N}(m, I_d)$ where $m=(3,0,\ldots,0)^{\mathsf{T}}\in\RR^d$, for $d=10, 50, 100$.
The Schr{\"o}dinger bridge should steer the diffusion from the origin to a Gaussian centered at $m$.

The drift network uses LayerNorm and hidden dimensions that scale with $d$. 
The hyperparameters are summarized in Table \ref{tab:sbp_config}.
The SDE is integrated via the Euler--Maruyama scheme on a uniform grid of $20$ points, that is, $19$ time steps.

\begin{table}[htbp]
\centering
\caption{Hyperparameters for the high-dimensional SBP experiments.}
\label{tab:sbp_config}
\begin{tabular}{ccccccc}
\toprule
$d$ & $M$ & $N$ & Epochs & Hidden dims & $\alpha$ & $\lambda^{-1}$ \\
\midrule
$10$  & $400$  & $160$ & $4000$ & $(128,64)$ & $0.1$  & $0.002$ \\
$50$  & $800$  & $160$ & $5000$ & $(256,128)$ & $0.02$ & $10^{-3}$ \\
$100$ & $1500$ & $160$ & $6000$ & $(512,256)$ & $0.01$ & $6\times 10^{-4}$ \\
\bottomrule
\end{tabular}
\end{table}

\begin{table}[htbp]
\centering
\caption{Results for the high-dimensional SBP (Gaussian shift, target mean $m_1=3.0$).
Values are mean $\pm$ standard deviation over $5$ independent runs with different random seeds, for each $d$.
All reported quantities are computed after training on an evaluation batch that is independent of the training batches
($1000$ paths for $d\le 50$ and $500$ paths for $d=100$), and the evaluation MMD${}^2$ uses $500$ random frequencies,
not the training value of $M$.}
\label{tab:sbp_result}
\begin{tabular}{ccccc}
\toprule
$d$ & $\EE[X_1]_1$ & $\frac{1}{d-1}\sum_{k=2}^{d}\EE[X_1]_k$ & Std (mean) & MMD${}^2$ \\
\midrule
$10$  & $2.69\pm 0.07$ & $\phantom{-}0.002\pm 0.025$ & $0.98\pm 0.00$ & $(1.4\pm 0.2)\times 10^{-3}$ \\
$50$  & $2.64\pm 0.07$ & $\phantom{-}0.002\pm 0.007$ & $0.96\pm 0.00$ & $(2.1\pm 0.2)\times 10^{-3}$ \\
$100$ & $2.71\pm 0.05$ & $-0.001\pm 0.002$ & $0.95\pm 0.00$ & $(2.0\pm 0.3)\times 10^{-3}$ \\
\bottomrule
\end{tabular}
\end{table}

The results are shown in Table \ref{tab:sbp_result}.
In all dimensions, the first coordinate $\EE[X_1]_1$ approaches the target $3.0$,
the remaining coordinates stay near zero,
and the standard deviation is close to the target value of $1.0$.
The MMD${}^2$ decreases to the order of $10^{-3}$, confirming the distributional constraint is approximately satisfied.
The standard deviations across seeds are small (between $0.05$ and $0.07$ for $\EE[X_1]_1$ in all three dimensions),
indicating that the method is stable and the reported performance is not seed-specific.
The accuracy is remarkably consistent across dimensions, with $\EE[X_1]_1$ in the range $2.64$--$2.71$ across $d=10, 50, 100$.

For the $d=10$ Gaussian target, the hard-constrained Schr{\"o}dinger bridge is available in closed form:
its first-coordinate mean is $3t$, every marginal variance is
$q_t=t^2+\sigma^2t(1-t)$, and its feedback is
$u^*(t,x)=\{(1-\sigma^2)x+\sigma^2m\}/\{\sigma^2+(1-\sigma^2)t\}$.
Figure~\ref{fig:gaussian_bridge_validation} compares this bridge with the learned finite-penalty solution
at $\lambda^{-1}=2\times10^{-4}$ over three seeds.
The learned control energy $\frac12\EE\int_0^1|u|^2dt$ is $6.494\pm0.033$, compared with the exact value
$6.517$; the terminal first-coordinate mean and standard deviation are $2.910\pm0.042$ and
$1.027\pm0.040$, respectively.  Thus the marginal flow and energy are close to the analytic bridge,
while the remaining mean gap is consistent with the fact that the network solves a finite-penalty problem rather than the hard terminal constraint.

\begin{figure}[htbp]
\centering
\includegraphics[width=\textwidth]{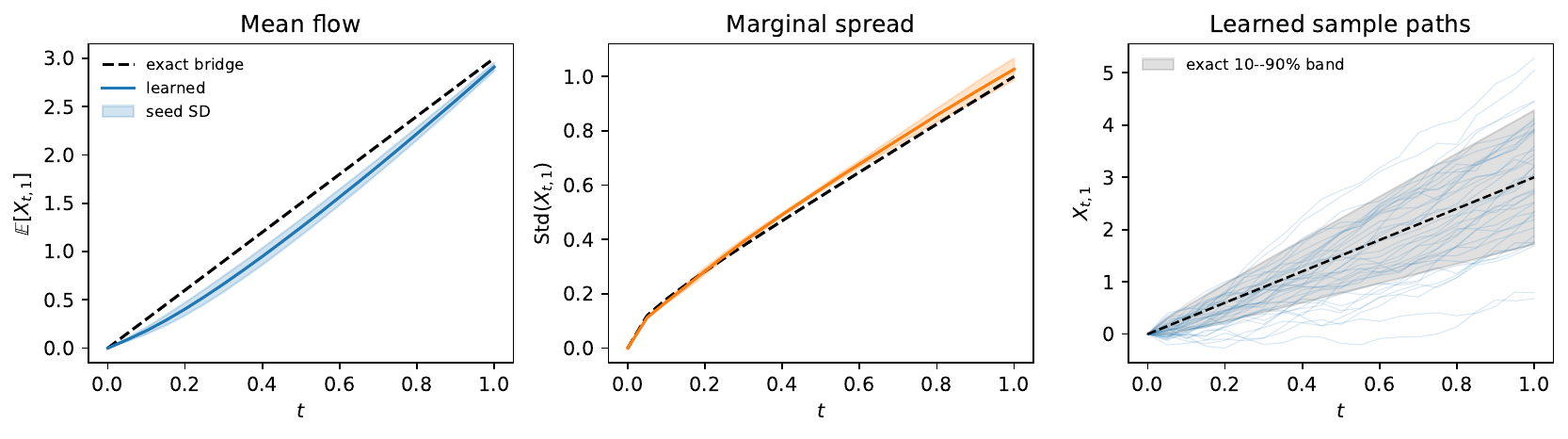}
\caption{Analytic Gaussian bridge validation in $d=10$ ($3$ seeds, $5000$ epochs).
Left and centre: learned first-coordinate marginal mean and standard deviation (solid; shading is one standard deviation across seeds) against the exact hard-constrained bridge (dashed).
Right: $40$ learned sample paths from one seed and the exact bridge's $10$--$90\%$ marginal band.}
\label{fig:gaussian_bridge_validation}
\end{figure}

The residual gap between $\EE[X_1]_1$ and the target $3.0$ has two sources.
First, with finite penalty weight $\lambda$ the optimal penalised solution does not exactly match $\mu_1$,
a gap that shrinks as $\lambda^{-1}\to 0$ (verified in Section~\ref{sec:5.4lam}).
Second, the bandwidth scaling $\alpha=1/d$ keeps the cosine argument
$\mathrm{Var}(Z^{\mathsf{T}}X\mid X)=2\alpha|X|^2$ at $O(1)$ for high-dimensional targets with $\EE|X|^2=O(d)$,
which is necessary for the RF features to remain informative and for the
$O(1/M)+O(1/N)$ variance bound of Theorem~\ref{thm:rfu_unbiased}(iii) to be free of hidden $d$-dependent constants. 
The price is a flatter kernel that reduces MMD discriminative power,
so the same MMD${}^2$ value corresponds to a larger distributional discrepancy in higher dimensions.
The residual gap $\EE[X_1]_1\approx 2.7$ in Table~\ref{tab:sbp_result}
is dominated by the finite $\lambda$ rather than by $\alpha=1/d$. 
At $\lambda^{-1}=2\times10^{-4}$ the same $d=100$ configuration reaches $\EE[X_1]_1=2.91$
(Section~\ref{sec:conv_ver}), which removes about $70\%$ of the residual gap $3.0-2.71$ reported here.
Overall, the experiments demonstrate that the proposed method enables Schr{\"o}dinger bridge computation
in dimensions well beyond the $d\le 2$ range of the $O(dN^2)$ kernel method in \cite{nak:2024sb}.

\subsection{Convergence verification}\label{sec:conv_ver}

The high-dimensional experiments of Section~\ref{sec:5.3} exhibit a residual gap
between $\EE[X_1]_1$ and the target $3.0$.
We now diagnose this gap experimentally by sweeping the two parameters
that control the approximation: the penalty weight $\lambda$ (Theorem~\ref{thm:convergence})
and the number of random Fourier features $M$ (Corollary~\ref{cor:kernel_convergence}).

\subsubsection{Penalty parameter sweep}\label{sec:5.4lam}

We illustrate the penalty trend predicted by Theorem~\ref{thm:convergence}(i),
\[
 \lim_{n\to\infty}\sqrt{\lambda_n}\,\gamma_K(\mathcal{L}(X_1^{u_n}),\mu_1) = 0,
\]
by sweeping the penalty parameter for the Gaussian shift target
($\mu_1=\mathcal{N}(m,I_d)$ with $m=(3,0,\ldots,0)^{\mathsf{T}}$) at $d=10$ and $d=100$.
All other settings (network architecture, bandwidth, $M$, $N$, Euler--Maruyama grid)
are kept identical to those of Section~\ref{sec:5.3}, except that every run uses $5000$ epochs.
We vary $\lambda^{-1}\in\{2\times10^{-2}, 6\times 10^{-3}, 2\times10^{-3}, 6\times 10^{-4}, 2\times10^{-4}\}$
and report mean and standard deviation over $3$ random seeds.

The configuration $\lambda^{-1}=2\times10^{-3}$ appears in three tables of this section under three evaluation protocols,
and the reported values differ accordingly:
$\EE[X_1]_1=2.69\pm0.07$ in Table~\ref{tab:sbp_result} ($4000$ epochs, $5$ seeds, evaluation by the RF $U$-statistic),
$2.58\pm0.06$ here ($5000$ epochs, the first $3$ seeds, same evaluation),
and $2.67\pm0.12$ in Table~\ref{tab:kernel_vs_rf} ($5000$ epochs, $5$ seeds, evaluation by the kernel $U$-statistic).
The three runs use identical hyperparameters, and the per-seed values range from $2.52$ to $2.83$,
so the differences are entirely due to the seed subset and, for the MMD${}^2$ column, to the evaluation estimator;
the training length has little effect over this range.

\begin{table}[htbp]
\centering
\caption{$\lambda$-sweep for the $d=10$ Gaussian shift SBP
($\mu_1=\mathcal{N}(m, I_{10})$, $m_1=3.0$).
Mean $\pm$ standard deviation over $3$ seeds, $5000$ epochs each.
The control cost denotes the empirical estimate of $\frac12\EE\bigl[\int_0^1|u|^2 dt\bigr]$ at the end of training.}
\label{tab:lambda_sweep}
\begin{tabular}{ccccc}
\toprule
$\lambda^{-1}$ & $\EE[X_1]_1$ & MMD${}^2$ & $|\EE[X_1]-m|^2$ & Control cost \\
\midrule
$2\times10^{-2}$     & $1.78\pm 0.05$ & $(2.4\pm 0.2)\times 10^{-2}$ & $1.51\pm 0.12$ & $3.7\pm 0.2$ \\
$6\times 10^{-3}$    & $2.28\pm 0.05$ & $(5.2\pm 1.2)\times 10^{-3}$ & $0.54\pm 0.08$ & $5.1\pm 0.2$ \\
$2\times10^{-3}$     & $2.58\pm 0.06$ & $(2.0\pm 0.7)\times 10^{-3}$ & $0.21\pm 0.05$ & $5.8\pm 0.2$ \\
$6\times 10^{-4}$    & $2.76\pm 0.06$ & $(1.1\pm 0.3)\times 10^{-3}$ & $0.087\pm 0.028$ & $6.5\pm 0.4$ \\
$2\times10^{-4}$     & $2.81\pm 0.06$ & $(0.8\pm 0.2)\times 10^{-3}$ & $0.061\pm 0.020$ & $6.9\pm 0.5$ \\
\bottomrule
\end{tabular}
\end{table}

The results in Table~\ref{tab:lambda_sweep} confirm the predicted convergence behaviour.
As $\lambda^{-1}$ decreases (i.e., $\lambda$ increases) by two orders of magnitude,
the terminal MMD${}^2$ decreases by approximately $30$-fold,
the squared mean error $|\EE[X_1]-m|^2$ decreases by approximately $25$-fold,
and the first-coordinate mean $\EE[X_1]_1$ approaches the target value $3.0$ from below.
The product $\lambda\cdot\mathrm{MMD}^2$, an upper estimate of $\lambda\gamma_K^2$, instead grows mildly
from $1.2$ at $\lambda^{-1}=2\times10^{-2}$ to $4.0$ at $\lambda^{-1}=2\times10^{-4}$;
equivalently the evaluation MMD${}^2$ decays like $\lambda^{-0.74}$ over the sweep,
slower than the $o(1/\lambda)$ asserted asymptotically by Theorem~\ref{thm:convergence}~(i).
This is expected rather than contradictory: the sweep is run at a fixed epoch budget
and need not produce population near-minimizers with $\varepsilon_n\to0$;
moreover, its fixed evaluation sizes do not satisfy the asymptotic conditions \eqref{eq:4.cond}.
What the sweep does verify is the penalty-bias term of Proposition~\ref{prop:rate}:
the terminal discrepancy decreases monotonically as $\lambda$ grows.
The control cost simultaneously increases monotonically from $3.7$ to $6.9$,
exhibiting the expected fidelity--cost trade-off:
stronger enforcement of the terminal constraint requires a more aggressive control.

The residual gap of $\EE[X_1]_1 \approx 2.81$ to the target $3.0$ at $\lambda^{-1}=2\times10^{-4}$
is no longer dominant, and could in principle be eliminated by further decreasing $\lambda^{-1}$ at the cost of additional iterations.
This experiment thus identifies the finite penalty weight as the principal source of the residual gap
discussed in Section~\ref{sec:5.3}.

To confirm that the same diagnosis applies in higher dimensions and that the residual gap
$\EE[X_1]_1\approx 2.7$ observed in Table~\ref{tab:sbp_result} for $d\ge 10$ is not a structural
limitation of the $\alpha=1/d$ scaling, we repeat the sweep at $d=100$
with the corresponding hyperparameters of Table~\ref{tab:sbp_config}
($\alpha=0.01$, $M=1500$, $N=160$, $6000$ epochs);
the analogous sweep at the intermediate $d=50$ is qualitatively identical
and is omitted to save space.

\begin{table}[htbp]
\centering
\caption{$\lambda$-sweep for the $d=100$ Gaussian shift SBP
($\mu_1=\mathcal{N}(m, I_{100})$, $m_1=3.0$).
Mean $\pm$ standard deviation over $3$ seeds, $6000$ epochs each, $\alpha=0.01$, $M=1500$.}
\label{tab:lambda_sweep_d100}
\begin{tabular}{ccccc}
\toprule
$\lambda^{-1}$ & $\EE[X_1]_1$ & MMD${}^2$ & $|\EE[X_1]-m|^2$ & Control cost \\
\midrule
$2\times10^{-2}$     & $0.96\pm 0.02$ & $(1.1\pm 0.0)\times 10^{-1}$ & $4.223\pm 0.099$ & $\phantom{0}2.2\pm 0.3$ \\
$6\times 10^{-3}$    & $1.72\pm 0.03$ & $(3.6\pm 0.1)\times 10^{-2}$ & $1.774\pm 0.089$ & $8.3\pm 0.7$ \\
$2\times10^{-3}$     & $2.28\pm 0.01$ & $(9.0\pm 0.4)\times 10^{-3}$ & $0.740\pm 0.014$ & $15.4\pm 1.0$ \\
$6\times 10^{-4}$    & $2.72\pm 0.06$ & $(2.0\pm 0.1)\times 10^{-3}$ & $0.321\pm 0.029$ & $21.7\pm 1.8$ \\
$2\times10^{-4}$     & $2.91\pm 0.01$ & $(1.3\pm 0.1)\times 10^{-3}$ & $0.263\pm 0.034$ & $25.9\pm 2.3$ \\
\bottomrule
\end{tabular}
\end{table}

Table~\ref{tab:lambda_sweep_d100} shows that the trend in $d=100$ is identical to and, if anything, more favourable than in $d=10$. 
$\EE[X_1]_1$ progresses from $0.96\pm 0.02$ at $\lambda^{-1}=2\times10^{-2}$
to $2.91\pm 0.01$ at $\lambda^{-1}=2\times10^{-4}$,
closing approximately $96\%$ of the gap between the value at $\lambda^{-1}=2\times10^{-2}$ and the target $3.0$.
The best value of $\EE[X_1]_1$ at $\lambda^{-1}=2\times10^{-4}$ thus lies in the narrow range $2.81$ (at $d=10$) to $2.91$ (at $d=100$),
demonstrating that the residual gap of $\EE[X_1]_1\approx 2.7$
reported at the working $\lambda^{-1}$ values of Table~\ref{tab:sbp_result}
is driven by the choice of $\lambda$ rather than by a degradation of the method with dimension.
The price of pushing $\lambda$ further is a roughly proportional increase in the control cost
(at $d=100$: from $2.2$ to $25.9$, a $12$-fold increase)
and in the number of optimization iterations needed for convergence,
a property of the penalized problem itself rather than of Proposition~\ref{prop:rate}.

\subsubsection{Kernel $U$-statistic limit ($M\to\infty$)}\label{sec:5.4ker}

We compare the random Fourier estimator with the kernel-$U$-statistic limit used in
Corollary~\ref{cor:kernel_convergence}.  For fixed data,
$\hat\gamma_{M,N}^2\to\bar\gamma_K^2(D)$ almost surely as $M\to\infty$.
Using the $d=10$ Gaussian shift setting of Section~\ref{sec:5.3} as a controlled benchmark
($N=160$, $\lambda^{-1}=2\times10^{-3}$, $5000$ epochs rather than the $4000$ of Table~\ref{tab:sbp_config}, $5$ seeds,
all other hyperparameters unchanged),
we compare four penalty estimators: the kernel $U$-statistic $\bar{\gamma}_K^2(D)$
of complexity $O(dN^2)$, and the RF $U$-statistic $\hat{\gamma}_{M,N}^2$ for $M\in\{100,400,1600\}$.
Evaluation is performed by the kernel $U$-statistic in all cases for fair comparison.

\begin{table}[htbp]
\centering
\caption{Kernel $U$-statistic vs RF $U$-statistic penalty in SBP training
($d=10$ Gaussian shift, $N=160$, $\lambda^{-1}=2\times10^{-3}$, $5000$ epochs, $5$ seeds).
Evaluation MMD${}^2$ is computed by the kernel $U$-statistic for fair comparison.}
\label{tab:kernel_vs_rf}
\begin{tabular}{lcccc}
\toprule
Method & $\EE[X_1]_1$ & $|\EE[X_1]-m|^2$ & MMD${}^2$ (eval) & Std (mean) \\
\midrule
Kernel $U$-stat       & $2.69\pm 0.08$ & $0.13\pm 0.04$ & $(1.40\pm 0.70)\times 10^{-3}$ & $0.98\pm 0.01$ \\
RF $U$-stat $(M=100)$ & $2.63\pm 0.09$ & $0.18\pm 0.07$ & $(2.11\pm 0.88)\times 10^{-3}$ & $0.98\pm 0.01$ \\
RF $U$-stat $(M=400)$ & $2.67\pm 0.12$ & $0.16\pm 0.07$ & $(1.98\pm 0.77)\times 10^{-3}$ & $0.98\pm 0.01$ \\
RF $U$-stat $(M=1600)$& $2.77\pm 0.09$ & $0.10\pm 0.05$ & $(2.07\pm 0.84)\times 10^{-3}$ & $0.97\pm 0.00$ \\
\bottomrule
\end{tabular}
\end{table}

The four methods produce statistically indistinguishable solutions. 
The first-coordinate means $\EE[X_1]_1$ all fall in $[2.63, 2.77]$, well within the seed-to-seed standard deviation,
and the standard deviations across the remaining coordinates agree to two decimal places.
The MMD${}^2$ values for the RF estimators ($M=100, 400, 1600$) cluster around $2\times 10^{-3}$,
slightly above the kernel $U$-statistic's $1.4\times 10^{-3}$,
but again within seed variability.
This empirical equivalence corroborates the fixed-data RF approximation underlying
Corollary~\ref{cor:kernel_convergence}: the RF estimator is a faithful surrogate
of the kernel $U$-statistic in the SBP training loop,
and increasing $M$ does not push the RF method beyond the kernel $U$-statistic baseline.

A complementary cost analysis is given in Section~\ref{sec:5.5comp},
where we show that for sufficiently large batch size $N$ the kernel $U$-statistic computation
($O(dN^2)$ per iteration, with backpropagation through the full $N\times N$ kernel matrix)
becomes a memory and compute bottleneck, whereas the RF $U$-statistic remains scalable
at $O(dNM)$ cost.
For the moderate $N=160$ used here, both estimators have comparable per-iteration cost,
so the table demonstrates only the equivalence in solution quality. 
The computational advantage of the RF method emerges in higher dimensions and larger batch sizes.

\subsection{Computational scaling}\label{sec:5.5comp}

We compare the proposed RF $U$-statistic with the kernel estimator \eqref{eq:empirical_mmd}
in terms of both solution quality and computational cost.

We measure the per-iteration wall-clock time of the MMD computation alone (forward + backward,
$d=10$, $\alpha=0.1$, $M=500$), which isolates the kernel-evaluation cost from the SDE overhead.

\begin{table}[htbp]
\centering
\caption{MMD computation time (forward + backward, in milliseconds; $d=10$, $\alpha=0.1$, $M=500$).
Solution quality of the two estimators is comparable in the SBP training pipeline
(verified in Section~\ref{sec:5.4ker} for $d=10$).}
\label{tab:scalability_mmd}
\begin{tabular}{cccc}
\toprule
$N$ & Kernel $O(dN^2)$ & RF $U$-stat $O(dNM)$ & Speedup \\
\midrule
$100$  & $0.6$  & $0.7$ & $0.8\times$ \\
$200$  & $0.8$  & $0.6$ & $1.3\times$ \\
$500$  & $4.7$  & $1.1$ & $4.1\times$ \\
$1000$ & $19.3$ & $2.0$ & $9.5\times$ \\
$2000$ & $83.3$ & $3.7$ & $22.3\times$ \\
\bottomrule
\end{tabular}
\end{table}

The speedup reaches $22\times$ at $N=2000$, confirming the $O(dN^2)$ vs $O(dNM)$ scaling
(equivalently $O(N^2)$ vs $O(NM)$ here because $d=10$ is fixed).
The kernel method also requires storing the $N\times N$ kernel matrix in the computation graph
for backpropagation through the SDE trajectory,
which becomes the binding memory constraint as $N$ grows;
the RF $U$-statistic avoids this entirely: in the vectorized implementation used here
the memory footprint is $O(NM + Nd)$, and a loop over the random frequencies reduces it to $O(N+Md)$
at the same $O(dNM)$ computational cost.
This memory advantage is essential for high-dimensional MFGs and SBPs
where both large $N$ (for accurate sampling) and large $d$ are required simultaneously.

\subsection{EV charging coordination}\label{sec:5ev}

We consider a fleet of EVs that must charge over $[0,T]$, with $T=1$.
The state is $X_t=(s_t,h)$, where $s_t\in[0,1]$ is the state of charge (SOC)
and the time-independent variable $h$ represents bounded vehicle heterogeneity.
We sample $h$ from $\mathcal N(0,0.3^2)$ truncated to $[-0.9,0.9]$ and set
\[
 \eta(h):=\frac{e^h}{\mathbb E[e^H]},
 \qquad H\sim\mathcal N(0,0.3^2)\ \text{conditioned on }|H|\le0.9,
\]
so that $\eta$ is positive, bounded, and normalized to have mean one.

The neural network outputs a bounded unidirectional charging fraction
$a_\theta(t,s,h)\in(0,1)$.  With $\ell(z)=(1+e^{-z})^{-1}$, define the
normalized SOC-dependent charging capability and the applied charging power by
\begin{equation}\label{eq:5ev.power}
 p_{\max}(s)=\ell(20(s-0.1))\ell(20(0.85-s)),\qquad
 P_\theta(t,s,h)=\eta(h)p_{\max}(s)a_\theta(t,s,h).
\end{equation}
The same $P_\theta$ drives the SOC and is counted as grid demand:
\begin{equation}\label{eq:5ev.dyn}
 ds_t=P_\theta(t,s_t,h)\,dt+0.02\,dB_t,\qquad dh_t=0,\qquad
 D_t=\mathbb E[P_\theta(t,s_t,h)].
\end{equation}
Since the same $P_\theta$ enters both the SOC dynamics and $D_t$, reducing $D_t$
necessarily reduces the fleet-mean applied charging power at time $t$.
The reported $D_t$ is fleet-mean
\emph{normalized} charging power.  If the reference charger rating is
$p_{\rm ref}$ kW and the fleet size is $n_{\rm fleet}$, the corresponding
mean per-vehicle and total powers are $p_{\rm ref}D_t$ and
$n_{\rm fleet}p_{\rm ref}D_t$, respectively; no particular kW rating is imposed here.

The numerical objective is
\begin{equation}\label{eq:5ev.objective}
 \frac12\mathbb E\!\left[\int_0^T P_\theta(t,s_t,h)^2dt\right]
 +c\int_0^TD_t^2dt
 +\lambda\gamma_K(\mathcal L(s_T),\mu_T)^2,
\end{equation}
where $K(s,s')=e^{-50(s-s')^2}$, $\lambda=10^3$, and $\mu_T$ is
$\mathcal N(0.85,0.05^2)$ truncated to $[0,1]$.
The initial SOC is $\mathcal N(0.20,0.05^2)$ truncated to $[0,1]$.
Set $Q_t:=D_t^2$. For powers $P_1,\ldots,P_N$, an unbiased $O(N)$
estimator of $Q_t$ is
\begin{equation}\label{eq:5ev.Rhat}
 \hat Q_t
 =\frac{1}{N(N-1)}\sum_{i\ne j}P_iP_j
 =\frac{(\sum_iP_i)^2-\sum_iP_i^2}{N(N-1)}.
\end{equation}
All factors in \eqref{eq:5ev.power} are bounded, so this estimator is bounded
and conditionally unbiased for $Q_t$.

\paragraph{Relation to the convergence theorem.}
This experiment is a physically consistent numerical extension, not a direct
application of Theorem~\ref{thm:convergence} or
Proposition~\ref{prop:convergence_general}.  The congestion term depends on the
state--control distribution through $P_\theta$, whereas those results assume a
functional of the state law alone.  In addition, $h$ is a frozen parameter,
the terminal MMD constrains only the SOC marginal, and the implementation projects
each Euler--Maruyama SOC update onto $[0,1]$.  Extending the convergence theorem
to this control-dependent, degenerate, constrained-state setting would require
additional analysis.  We therefore use the experiment only to test whether the
same sample-based optimization architecture produces a meaningful operational
trade-off after enforcing physical consistency.

Training uses $N=160$ paths, $M=400$ terminal-kernel frequencies, 24 time steps,
1500 epochs, and $c\in\{0,100,300,1000\}$.  Each trained policy is evaluated on
2000 held-out paths, and Table~\ref{tab:ev_charging} reports the mean and sample
standard deviation over five seeds.  The service-failure event is $s_T<0.80$;
the target law itself has failure probability approximately $15.9\%$.

\begin{table}[htbp]
\centering
\caption{Physically consistent EV charging experiment.  Peak $D$ is
$\max_tD_t$ on held-out paths and is normalized rather than measured in kW.
Peak reduction is relative to $c=0$.}
\label{tab:ev_charging}
\begin{small}
\begin{tabular}{ccccccc}
\toprule
$c$ & $\mathbb E[s_T]$ & SOC std. & Eval.\ MMD${}^2$ & Peak $D$ & Reduction & $\mathbb P(s_T<0.80)$ \\
\midrule
$0$    & $0.841\pm0.004$ & $0.062$ & $(1.8\pm1.8)\times10^{-3}$ & $0.744\pm0.021$ & --      & $15.96\pm1.74\%$ \\
$100$  & $0.833\pm0.005$ & $0.079$ & $(2.7\pm2.3)\times10^{-3}$ & $0.692\pm0.009$ & $6.98\%$  & $17.59\pm1.79\%$ \\
$300$  & $0.794\pm0.011$ & $0.168$ & $(11.2\pm4.9)\times10^{-3}$ & $0.640\pm0.016$ & $13.89\%$ & $21.62\pm2.19\%$ \\
$1000$ & $0.691\pm0.021$ & $0.267$ & $(80.1\pm24.8)\times10^{-3}$ & $0.519\pm0.019$ & $30.25\%$ & $30.93\pm3.28\%$ \\
\bottomrule
\end{tabular}
\end{small}
\end{table}

At $c=100$, peak normalized power falls by about $7\%$, while the service-failure
rate increases by only $1.63$ percentage points relative to $c=0$.
At $c=300$ and $1000$, the peak reductions increase to about $14\%$ and $30\%$,
but the service-failure rates increase by $5.66$ and $14.97$ percentage points.
The result is therefore a Pareto trade-off between the grid-side peak and the
customer-side terminal service level, not a uniform improvement as $c$ grows.
Across all saved runs, $0<a_\theta<1$ and $0\le s_t\le1$.
To verify the implementation, we recomputed $P_\theta$ in
\eqref{eq:5ev.power} from every stored tuple
$(s_t,h,a_\theta(t,s_t,h))$ and compared it with the power used in the SOC
update and $D_t$.  The maximum discrepancy is below $2.4\times10^{-7}$,
consistent with single-precision rounding.  The boundary-projection rate is the
fraction of pre-projection Euler--Maruyama SOC updates outside $[0,1]$.
Of the $5\times2000\times24=240{,}000$ evaluation updates for each $c$,
respectively $0,0,2$, and $16$ are projected for $c=0,100,300$, and $1000$,
giving rates $0,0,8.3\times10^{-6}$, and $6.7\times10^{-5}$.  Thus the
projection is too rare to drive the observed trade-off.

\subsection{Positioning against deep-learning MFG / SBP baselines}\label{sec:5.7baselines}

The methodology developed in this paper is one of several deep-learning approaches to high-dimensional MFG and SBP that have appeared in recent years.
The most directly relevant baselines are
APAC-Net of Lin et al.~\cite{lin-etal:2021apac}, which alternates between a population network and a control network in a saddle-point formulation;
the HJB-based machine-learning framework of Ruthotto et al.~\cite{rut-etal:2020}, which parametrises the value function rather than the control;
the deep MFG/MFC methods surveyed in \cite{car-lau:2022,car-lau:2021};
and the bridge-matching and diffusion-Schr{\"o}dinger-bridge methods
\cite{de-etal:2021,var-etal:2021,shi-etal:2023dsbm,Tong2024}, which target the $c=0$ specialization.
Table~\ref{tab:baseline_comparison} summarises the structural differences relevant to a JSC-style assessment.

\begin{table}[htbp]
\centering
\caption{Structural comparison with representative deep-learning MFG and SBP methods.
``Networks'' counts the neural networks involved in the optimization;
``Cost class'' indicates the running- and terminal-cost structure each method supports;
``Convergence'' summarizes the theoretical level addressed;
``Per-iter'' is the asymptotic per-iteration cost scaling
($M$ = random-feature count, $N$ = batch size, $q$ = SDE time intervals,
$d$ = state dimension, $P_\theta$ = forward-plus-backward network cost per state--time input, up to a constant).
For the baselines, method-specific derivative and PDE-residual costs are implementation-dependent and are not included.}
\label{tab:baseline_comparison}
\resizebox{\textwidth}{!}{%
\begin{tabular}{lccccc}
\toprule
Method & Reference & Networks & Cost class & Convergence & Per-iter \\
\midrule
\textbf{Ours} & this paper & $1$ (drift) & MMD + nonnegative wk.\ cts.\ $\mathcal{R}$ & population + held-out a.s. & $O(qN(P_\theta+dM))$ \\
APAC-Net & \cite{lin-etal:2021apac} & $2$ (pop.\,+\,ctrl.) & Lagrangian, no soft target & PDE-level (asymptotic) & $O(qNP_\theta)$ \\
ML-MFG & \cite{rut-etal:2020} & $1$ (value fn.) & local + interaction (HJB form) & PDE-level (asymptotic) & $O(qNP_\theta)$ \\
Deep MFG/MFC & \cite{car-lau:2022,car-lau:2021} & $2$--$3$ & general $f(\cdot,\nu)$, $g(\cdot,\nu)$ & population-level & $O(qNP_\theta)$ \\
DSB / DSBM & \cite{de-etal:2021,shi-etal:2023dsbm} & $2$ (score) & SBP only ($c=0$) & empirical & $O(qNP_\theta)$ \\
SB max-lik. & \cite{var-etal:2021} & $2$ & SBP only ($c=0$) & empirical & $O(qNP_\theta)$ \\
\bottomrule
\end{tabular}%
}
\end{table}

Three structural distinctions are worth highlighting.
\emph{First}, our method uses a single drift network, whereas APAC-Net, DSB, DSBM and the deep MFG/MFC variants use two or more (e.g., a control network and a population/score network, or a value function and a forward-backward pair).
The reason is structural.
The terminal MMD penalty $\gamma_K^2(\mathcal{L}(X_T^{(n)}),\mu_T)$ is a functional of the controlled marginal $\mathcal{L}(X_T^{(n)})$
that can be estimated directly from sampled controlled paths,
without learning the density or score of $\mathcal{L}(X_T^{(n)})$.
In contrast, deep-learning MFG/SBP methods that target a hard distributional constraint or a score-matching objective
need an explicit estimate of the population density or score,
which is provided by an auxiliary population or score network.
\emph{Second}, Theorem~\ref{thm:convergence} treats population penalized
near-minimizers, while Corollary~\ref{cor:empirical_evaluation} gives almost-sure
consistency on an independent evaluation batch.
\emph{Third}, our per-iteration cost on one kernel-MMD evaluation is the explicit $O(dNM)$ of the random-feature representation;
the comparison methods rely on neural-network forward/backward passes whose cost is implicit in $P_\theta$.

\paragraph{Direct numerical comparison.}
A fair head-to-head benchmark with the above baselines is non-trivial because of differing problem formulations:
APAC-Net and Ruthotto et al.\ solve HJB-style problems without a soft distributional terminal target,
and DSB / DSBM target the SBP specialisation $c=0$ via score-matching rather than a penalty.
The deep MFG/MFC framework of \cite{car-lau:2022,car-lau:2021} admits per-agent costs $f(t,x,\alpha,\nu)$, $g(x,\nu)$
and in fact subsumes our kernel-MMD terminal cost as a special $g(x,\nu)$;
adopting that methodology would, however, replace our unbiased $O(dNM)$ estimator
by a biased sample mean and introduce a separate population or value-function network.
The generalised conditional-gradient method of \cite{nakam-sai:2026loc,nakam-sai:2026disc}
proves non-asymptotic convergence rates for potential MFGs with local coupling
on fully discretised finite-difference grids, which is complementary in scope:
their analysis is sharp on the grid but inherits the curse of dimensionality,
whereas our population penalty bound and held-out estimator bound
(Proposition~\ref{prop:rate_mfg}) are mesh-free and apply to nonlocal
kernel-MMD couplings.  Proposition~\ref{prop:convergence_general} also covers
non-kernel functionals of the state law, but not the control-dependent EV extension.
A quantitative benchmark on a common instance well suited to all candidate methods is left for future work;
Section~\ref{sec:5.5comp} gives the within-framework comparison of the RF $U$-statistic against the $O(dN^2)$ kernel method of \cite{nak:2024sb}.

\section{Proofs}\label{sec:6}

\subsection{Proof of Theorem \ref{thm:rfu_unbiased}}\label{sec:6.1}

We prove each part separately.

\noindent
(i) Unbiasedness.
Fix $z\in\RR^d$.
By expanding the squares in \eqref{eq:3.2},
\begin{align*}
 (S_c^X(z))^2 + (S_s^X(z))^2
 &= \sum_{i=1}^N\sum_{j=1}^N\bigl(\cos(z^{\mathsf{T}}X_i)\cos(z^{\mathsf{T}}X_j) + \sin(z^{\mathsf{T}}X_i)\sin(z^{\mathsf{T}}X_j)\bigr) \\
 &= \sum_{i=1}^N\sum_{j=1}^N\cos(z^{\mathsf{T}}(X_i-X_j)) = N + \sum_{\substack{i,j=1\\i\neq j}}^N\cos(z^{\mathsf{T}}(X_i-X_j)),
\end{align*}
where we used $\cos(a-b)=\cos a\cos b + \sin a\sin b$ and $\cos(0)=1$.
Since $X_1,\ldots,X_N$ are i.i.d.\ from $\mu$ and $X,X'$ denote independent copies with law $\mu$,
\[
 \sum_{\substack{i,j=1\\i\neq j}}^N\EE\left[\cos(z^{\mathsf{T}}(X_i-X_j))\right]
 = N(N-1)\EE[\cos(z^{\mathsf{T}}(X-X'))].
\]
Therefore
\begin{equation}\label{eq:6.1}
 \EE[U_{XX}(z)]
 = \frac{\EE[(S_c^X(z))^2 + (S_s^X(z))^2 - N]}{N(N-1)}
 = \EE[\cos(z^{\mathsf{T}}(X-X'))].
\end{equation}
Similarly, $\EE[U_{YY}(z)] = \EE[\cos(z^{\mathsf{T}}(Y-Y'))]$ where $Y,Y'$ are independent copies with law $\nu$.

For the cross term \eqref{eq:3.3}, since $(X_i, Y_j)$ are independent for all pairs,
\begin{align*}
 \EE[V_{XY}(z)]
 &= \frac{1}{N^2}\sum_{i=1}^N\sum_{j=1}^N \EE[\cos(z^{\mathsf{T}}(X_i-Y_j))] = \EE[\cos(z^{\mathsf{T}}(X-Y))].
\end{align*}
Therefore, for each fixed $z$,
\[
 \EE[U_{XX}(z) - 2V_{XY}(z) + U_{YY}(z)]
 = \EE[\cos(z^{\mathsf{T}}(X-X'))] - 2\EE[\cos(z^{\mathsf{T}}(X-Y))] + \EE[\cos(z^{\mathsf{T}}(Y-Y'))].
\]
Now, averaging over $Z_1,\ldots,Z_M$ and applying the tower property, we have 
\begin{align*}
 \EE[\hat{\gamma}_{M,N}^2]
 &= \frac{\Phi(0)}{M}\sum_{r=1}^M\EE\bigl[\EE[U_{XX}(Z_r) - 2V_{XY}(Z_r) + U_{YY}(Z_r)\,|\, Z_r]\bigr] \\
 &= \Phi(0)\,\EE\bigl[\EE[\cos(Z^{\mathsf{T}}(X-X'))\,|\, Z] - 2\EE[\cos(Z^{\mathsf{T}}(X-Y))\,|\, Z] + \EE[\cos(Z^{\mathsf{T}}(Y-Y'))\,|\, Z]\bigr] \\
 &= \gamma_K(\mu,\nu)^2,
\end{align*}
where the last step uses the representation \eqref{eq:repre}.

\noindent
(ii) Complexity.
For each frequency $z=Z_r$, forming the $N$ inner products $z^{\mathsf{T}}X_i$ costs $O(dN)$ operations;
the trigonometric evaluations and sums defining $S_c^X(z)$ and $S_s^X(z)$ then cost $O(N)$.
The same complexity holds for $S_c^Y(z)$ and $S_s^Y(z)$.
From these four sums, $U_{XX}(z)$, $V_{XY}(z)$, and $U_{YY}(z)$ are each computed in $O(1)$.
Since there are $M$ frequencies, the total cost is $O(dNM)$, or $O(NM)$ when $d$ is fixed.

\medskip\noindent
(iii) Variance.
Let
\[
 g(z) := \Phi(0)\bigl[U_{XX}(z) - 2V_{XY}(z) + U_{YY}(z)\bigr],
\]
so that $\hat{\gamma}_{M,N}^2 = \frac{1}{M}\sum_{r=1}^M g(Z_r)$.
Note that $g(z)$ depends on the data $D := (X_1,\ldots,X_N,Y_1,\ldots,Y_N)$,
which is shared across all frequencies $Z_1,\ldots,Z_M$.
Hence $g(Z_1),\ldots,g(Z_M)$ are conditionally independent given $D$.
Conditioning on $D$ via the law of total variance, we get 
\begin{equation}\label{eq:6.2}
 \mathrm{Var}(\hat{\gamma}_{M,N}^2)
 = \EE\!\left[\mathrm{Var}\!\left(\hat{\gamma}_{M,N}^2 \,|\, D\right)\right]
 + \mathrm{Var}\!\left(\EE\!\left[\hat{\gamma}_{M,N}^2 \,|\, D\right]\right).
\end{equation}

Let us observe the conditional term.
Given $D$, the frequencies $Z_1,\ldots,Z_M$ are i.i.d., so
\[
 \mathrm{Var}(\hat{\gamma}_{M,N}^2\,|\, D) = \frac{1}{M}\,\mathrm{Var}(g(Z_1)\,|\, D).
\]
From \eqref{eq:3.2}--\eqref{eq:3.3} and the trivial bounds $|U_{XX}(z)|, |V_{XY}(z)|, |U_{YY}(z)| \le 1$,
$|g(z)|\le 4\Phi(0)$ uniformly in $z$ and $D$.
Hence $\mathrm{Var}(g(Z_1) \mid D) \le 16\Phi(0)^2$, and therefore
$\EE[\mathrm{Var}(\hat{\gamma}_{M,N}^2\,|\, D)] = O(1/M)$.

As for the marginal term, a direct computation of the conditional expectation $\EE[g(Z_1)\,|\, D]$ gives 
\begin{equation}
\label{eq:6.tower}
\begin{aligned}
 \EE[g(Z_1)\,|\, D]
 &= \Phi(0)\EE\bigl[U_{XX}(Z) - 2V_{XY}(Z) + U_{YY}(Z) \bigm| D\bigr]  \\
 &= \frac{1}{N(N-1)}\sum_{\substack{i,j=1\\i\neq j}}^N K(X_i,X_j)
   - \frac{2}{N^2}\sum_{i,j=1}^N K(X_i,Y_j)
   + \frac{1}{N(N-1)}\sum_{\substack{i,j=1\\i\neq j}}^N K(Y_i,Y_j) \\
 &=: \bar{\gamma}_K^2(D), 
\end{aligned}
\end{equation}
where we used $\Phi(0)\EE[\cos(Z^{\mathsf{T}}(x-y))] = \Phi(x-y) = K(x,y)$ by \eqref{eq:2.3}.
Thus $\EE[\hat{\gamma}_{M,N}^2\,|\, D] = \bar{\gamma}_K^2(D)$ coincides with
the kernel $U$-statistic estimator \eqref{eq:empirical_mmd}.
It remains to bound $\mathrm{Var}(\bar{\gamma}_K^2(D))$.
Write
\[
 \bar{\gamma}_K^2(D) = T_{XX} - 2 T_{XY} + T_{YY},
\]
where
\begin{align*}
 T_{XX} &= \frac{1}{N(N-1)}\sum_{i\neq j} K(X_i,X_j),\quad
 T_{YY} = \frac{1}{N(N-1)}\sum_{i\neq j} K(Y_i,Y_j),\\
 T_{XY} &= \frac{1}{N^2}\sum_{i,j=1}^{N} K(X_i,Y_j).
\end{align*}
The terms $T_{XX}$ and $T_{YY}$ are one-sample $U$-statistics of degree $2$
with symmetric kernel $K(\cdot,\cdot)$,
and $T_{XY}$ is a two-sample $U$-statistic of degree $(1,1)$ with kernel $K(\cdot,\cdot)$.

By Hoeffding's variance formula for one-sample $U$-statistics
\cite[Section~5.2]{ser:1980},
\[
 \mathrm{Var}(T_{XX}) = \frac{2}{N(N-1)}\bigl[2(N-2)\,\zeta_1^{X} + \zeta_2^{X}\bigr],
\]
where the Hoeffding variance components are
\[
 \zeta_1^{X} = \mathrm{Var}\bigl(\EE[K(X_1,X_2)\mid X_1]\bigr),\quad
 \zeta_2^{X} = \mathrm{Var}\bigl(K(X_1,X_2)\bigr).
\]
The uniform bound $|K(x,y)|\le \Phi(0)$ yields $\zeta_1^{X},\,\zeta_2^{X}\le \Phi(0)^2$,
hence $\mathrm{Var}(T_{XX}) \le 4\Phi(0)^2/N$ for $N\ge 2$.
The same bound applies to $T_{YY}$ with corresponding components $\zeta_1^{Y}, \zeta_2^{Y}\le\Phi(0)^2$.

To estimate the cross term $T_{XY}$, we use the representation
\[
 \mathrm{Var}(T_{XY}) = \mathrm{Var}\bigl(\EE[T_{XY}\,|\, X]\bigr) + \EE\bigl[\mathrm{Var}(T_{XY}\,|\, X)\bigr].
\]
Setting $\bar{K}_{\nu}(x):=\int_{\RR^d}K(x,y)\,\nu(dy)$, we have
$\EE[T_{XY}\,|\, X] = \frac{1}{N}\sum_{i=1}^N\bar{K}_{\nu}(X_i)$, so that
\[
 \mathrm{Var}\bigl(\EE[T_{XY}\,|\, X]\bigr)
 = \frac{1}{N}\mathrm{Var}\bigl(\bar{K}_{\nu}(X_1)\bigr)
 \le \frac{\Phi(0)^2}{N},
\]
where we used $|\bar{K}_{\nu}(x)|\le \Phi(0)$.
As for the conditional variance term, given $X$ the random variables $Y_1,\ldots,Y_N$ are i.i.d., so we have 
\[
 \mathrm{Var}(T_{XY}\,|\, X)
 = \frac{1}{N^2}\sum_{j=1}^N\mathrm{Var}\left(\left.\frac{1}{N}\sum_{i=1}^N K(X_i,Y_j)\,\right|\, X\right)
 \le \frac{\Phi(0)^2}{N},
\]
since the quantity inside the variance is bounded by $\Phi(0)$ in absolute value.
Therefore, $\mathrm{Var}(T_{XY}) \le 2\Phi(0)^2/N$.

Combining these three estimates via the triangle inequality for the $L^2$-norm, we see 
\[
 \sqrt{\mathrm{Var}(\bar{\gamma}_K^2(D))}
 \le \sqrt{\mathrm{Var}(T_{XX})} + 2\sqrt{\mathrm{Var}(T_{XY})} + \sqrt{\mathrm{Var}(T_{YY})}
 \le \frac{C\,\Phi(0)}{\sqrt{N}}
\]
for some constant $C > 0$, whence $\mathrm{Var}(\bar{\gamma}_K^2(D)) = O(1/N)$.
Consequently, substituting into \eqref{eq:6.2}, we obtain
\[
 \mathrm{Var}(\hat{\gamma}_{M,N}^2) = O(1/M) + O(1/N), 
\]
completing the proof of the theorem. 

\begin{rem}\label{rem:6.var}
The decompositions \eqref{eq:6.2} and \eqref{eq:6.tower} show that the proposed estimator
is exactly the kernel $U$-statistic $\bar{\gamma}_K^2(D)$ averaged over random frequencies:
the random Fourier representation contributes only the conditional variance term $O(1/M)$,
while the irreducible sample variance $O(1/N)$ inherent to the kernel $U$-statistic is preserved.
This decomposition is also consistent with Lemma~{\rm \ref{lem:strong_consistency}}:
both the kernel $U$-statistic $\bar{\gamma}_K^2(D)\to\gamma_K^2$ as $N\to\infty$,
and the random-feature average $\hat{\gamma}_{M,N}^2 - \bar{\gamma}_K^2(D)\to 0$, $M\to\infty$,
contribute to the joint a.s.\ convergence used in Corollary~{\rm \ref{cor:empirical_evaluation}}.
\end{rem}

\subsection{Proof of Theorem~\ref{thm:rfu_interaction}}\label{sec:6.1interaction}

The proof follows the same template as that of Theorem~\ref{thm:rfu_unbiased} given in Section~\ref{sec:6.1}.
Since the interaction $\mathcal{R}[\nu]$ involves only the $X$-sample,
the cross terms $V_{XY}$ and $U_{YY}$ are absent, and $(K,\Phi,\rho)$ is replaced by $(W,\Psi,\rho_W)$.
We prove each part separately.

\noindent
(i) Unbiasedness.
Fix $\tilde z\in\RR^d$.
By the algebraic identity used in the derivation of \eqref{eq:3.2},
\[
 (S_c^X(\tilde z))^2 + (S_s^X(\tilde z))^2
 = N + \sum_{\substack{i,j=1\\i\neq j}}^N\cos(\tilde z^{\mathsf{T}}(X_i-X_j)).
\]
Since $X_1,\ldots,X_N$ are i.i.d.\ from $\nu$ and $X,X'$ denote independent copies with law $\nu$,
\[
 \sum_{\substack{i,j=1\\i\neq j}}^N\EE\bigl[\cos(\tilde z^{\mathsf{T}}(X_i-X_j))\bigr]
 = N(N-1)\,\EE\bigl[\cos(\tilde z^{\mathsf{T}}(X-X'))\bigr],
\]
and therefore
\begin{equation}\label{eq:6.UXX_W}
 \EE[U_{XX}(\tilde z)]
 = \frac{\EE[(S_c^X(\tilde z))^2+(S_s^X(\tilde z))^2-N]}{N(N-1)}
 = \EE\bigl[\cos(\tilde z^{\mathsf{T}}(X-X'))\bigr].
\end{equation}
Averaging over $\tilde Z_1,\ldots,\tilde Z_M\sim\rho_W/\Psi(0)$ and applying the tower property,
\begin{align*}
 \EE\bigl[\hat{\mathcal{R}}_{M,N}[\hat\nu^N]\bigr]
 &= \frac{\Psi(0)}{2M}\sum_{r=1}^M
    \EE\bigl[\EE[U_{XX}(\tilde Z_r)\,|\,\tilde Z_r]\bigr] \\
 &= \frac{\Psi(0)}{2}\,\EE\bigl[\EE[\cos(\tilde Z^{\mathsf{T}}(X-X'))\,|\,\tilde Z]\bigr]
 = \mathcal{R}[\nu],
\end{align*}
where the last equality uses the Fourier representation \eqref{eq:repre_R}.

\medskip\noindent
(ii) Complexity.
For each frequency $\tilde z=\tilde Z_r$, forming the $N$ inner products
$\tilde z^{\mathsf{T}}X_i$ costs $O(dN)$ operations; the trigonometric evaluations
and sums defining $S_c^X(\tilde z)$ and $S_s^X(\tilde z)$ then cost $O(N)$.
From these two sums, $U_{XX}(\tilde z)$ is computed in $O(1)$.
Since there are $M$ frequencies and assembling \eqref{eq:3.5} adds an $O(M)$ overhead,
the total cost is $O(dNM)$, or $O(NM)$ when $d$ is fixed.

\medskip\noindent
(iii) Variance.
Let
\[
 g_W(\tilde z) := \frac{\Psi(0)}{2}\,U_{XX}(\tilde z),
\]
so that $\hat{\mathcal{R}}_{M,N}[\hat\nu^N]=\frac{1}{M}\sum_{r=1}^M g_W(\tilde Z_r)$.
Note that $g_W(\tilde z)$ depends on the data $D_X:=(X_1,\ldots,X_N)$,
which is shared across all frequencies $\tilde Z_1,\ldots,\tilde Z_M$.
Hence $g_W(\tilde Z_1),\ldots,g_W(\tilde Z_M)$ are conditionally independent given $D_X$.
Conditioning on $D_X$ via the law of total variance, we get
\begin{equation}\label{eq:6.var_W}
 \mathrm{Var}\bigl(\hat{\mathcal{R}}_{M,N}[\hat\nu^N]\bigr)
 = \EE\!\left[\mathrm{Var}\!\left(\hat{\mathcal{R}}_{M,N}[\hat\nu^N]\,|\,D_X\right)\right]
 + \mathrm{Var}\!\left(\EE\!\left[\hat{\mathcal{R}}_{M,N}[\hat\nu^N]\,|\,D_X\right]\right).
\end{equation}

Let us observe the conditional term.
Given $D_X$, the frequencies $\tilde Z_1,\ldots,\tilde Z_M$ are i.i.d., so
\[
 \mathrm{Var}\bigl(\hat{\mathcal{R}}_{M,N}[\hat\nu^N]\,|\,D_X\bigr)
 = \frac{1}{M}\,\mathrm{Var}(g_W(\tilde Z_1)\,|\,D_X).
\]
From the bound $|U_{XX}(\tilde z)|\le 1$, we have $|g_W(\tilde z)|\le \Psi(0)/2$ uniformly in $\tilde z$ and $D_X$.
Hence $\mathrm{Var}(g_W(\tilde Z_1)\,|\,D_X)\le \Psi(0)^2/4$, and therefore
$\EE[\mathrm{Var}(\hat{\mathcal{R}}_{M,N}[\hat\nu^N]\,|\,D_X)] = O(1/M)$.

As for the marginal term,
a direct computation of the conditional expectation $\EE[g_W(\tilde Z_1)\,|\,D_X]$ gives
\begin{equation}\label{eq:6.tower_W}
\begin{aligned}
 \EE[g_W(\tilde Z_1)\,|\,D_X]
 &= \frac{\Psi(0)}{2}\,\EE\bigl[U_{XX}(\tilde Z)\bigm|D_X\bigr] \\
 &= \frac{1}{2N(N-1)}\sum_{\substack{i,j=1\\i\neq j}}^N W(X_i,X_j)
 \;=:\; \bar{\mathcal{R}}_W(D_X),
\end{aligned}
\end{equation}
where we used $\Psi(0)\EE[\cos(\tilde Z^{\mathsf{T}}(x-y))]=\Psi(x-y)=W(x,y)$
by the analogue of \eqref{eq:2.3} for $W$.
Thus $\EE[\hat{\mathcal{R}}_{M,N}[\hat\nu^N]\,|\,D_X]=\bar{\mathcal{R}}_W(D_X)$ coincides with the
kernel $U$-statistic estimator of $\mathcal{R}[\nu]$.
It remains to bound $\mathrm{Var}(\bar{\mathcal{R}}_W(D_X))$.

Now $\bar{\mathcal{R}}_W(D_X)$ is a one-sample $U$-statistic of degree $2$ with symmetric kernel $\tfrac{1}{2}W(\cdot,\cdot)$.
By Hoeffding's variance formula for one-sample $U$-statistics
\cite[Section~5.2]{ser:1980},
\[
 \mathrm{Var}(\bar{\mathcal{R}}_W(D_X))
 = \frac{2}{N(N-1)}\bigl[2(N-2)\,\zeta_1 + \zeta_2\bigr],
\]
where the Hoeffding variance components are
\[
 \zeta_1 = \mathrm{Var}\!\left(\EE\!\left[\tfrac{1}{2}W(X_1,X_2)\bigm|X_1\right]\right),\qquad
 \zeta_2 = \mathrm{Var}\!\left(\tfrac{1}{2}W(X_1,X_2)\right).
\]
The uniform bound $|W(x,y)|\le \Psi(0)$ yields $\zeta_1,\,\zeta_2\le \Psi(0)^2/4$,
hence $\mathrm{Var}(\bar{\mathcal{R}}_W(D_X))\le \Psi(0)^2/N$ for $N\ge 2$.

Consequently, substituting into \eqref{eq:6.var_W}, we obtain
\[
 \mathrm{Var}\bigl(\hat{\mathcal{R}}_{M,N}[\hat\nu^N]\bigr) = O(1/M) + O(1/N),
\]
completing the proof of the theorem.
\qed

\subsection{Proof of Lemma~\ref{lem:strong_consistency}}\label{sec:6.LemSC}

Decompose $\hat{\gamma}_{M_n,N_n}^2 = \bar{\gamma}_K^2(D_{N_n}) + B_{M_n}$
as in the proof of Theorem~\ref{thm:rfu_unbiased}(iii) in Section~\ref{sec:6.1},
where $D_{N_n}=(X_1,\ldots,X_{N_n},Y_1,\ldots,Y_{N_n})$ and
\[
 B_{M_n} := \frac{1}{M_n}\sum_{r=1}^{M_n}\bigl[g(Z_r) - \bar{\gamma}_K^2(D_{N_n})\bigr].
\]

By Hoeffding's theorem on the strong law of large numbers for two-sample $U$-statistics
\cite[Section~5.4]{ser:1980},
$\bar{\gamma}_K^2(D_{N_n})\to \gamma_K(\mu,\nu)^2$ as $n\to\infty$, a.s.

Since cosines and sines are bounded by $1$, the Cauchy--Schwarz inequality
yields $S_c(z)^2+S_s(z)^2\le N^2$.
Combined with \eqref{eq:3.2}--\eqref{eq:3.3}, this gives $|U_{XX}(z)|, |V_{XY}(z)|, |U_{YY}(z)|\le 1$ for every $z\in\RR^d$ and every realization of the data,
so
\[
 |g(z)| \le \Phi(0)\bigl(|U_{XX}(z)|+2|V_{XY}(z)|+|U_{YY}(z)|\bigr)\le 4\Phi(0),
\]
uniformly in $z$ and $D_{N_n}$.
In particular $|\bar{\gamma}_K^2(D_{N_n})| = |\EE[g(Z_1)\mid D_{N_n}]|\le 4\Phi(0)$.

Fix a realization of $D_{N_n}$ and define
\[
 \xi_r := g(Z_r) - \bar{\gamma}_K^2(D_{N_n}), \quad r=1,\ldots,M_n.
\]
Since $Z_1,\ldots,Z_{M_n}$ are i.i.d.\ and independent of $D_{N_n}$,
the $\xi_r$ are, conditionally on $D_{N_n}$, i.i.d.\ with mean zero,
and take values in the interval
\[
 J(D_{N_n}) := \bigl[-4\Phi(0)-\bar{\gamma}_K^2(D_{N_n}),\;4\Phi(0)-\bar{\gamma}_K^2(D_{N_n})\bigr]\subset \RR,
\]
which has the length $8\Phi(0)$.
Hoeffding's inequality \cite[Theorem~2]{hoe:1963} for bounded i.i.d.\ centered random variables therefore gives,
for every $t>0$,
\[
 \mathbb{P}\left(\left.\left|\sum_{r=1}^{M_n}\xi_r\right|\ge t \,\right|\, D_{N_n}\right)
 \le 2\exp\left(-\frac{2 t^2}{M_n(8\Phi(0))^2}\right).
\]
Let $\varepsilon>0$. Setting $t=M_n\varepsilon$ and dividing inside the probability by $M_n$, we get
\[
 \mathbb{P}\left(|B_{M_n}|\ge\varepsilon \,|\, D_{N_n}\right)
 \le 2\exp\left(-\frac{M_n\varepsilon^2}{32\Phi(0)^2}\right).
\]
The right-hand side does not depend on the realization of $D_{N_n}$, so taking expectation over $D_{N_n}$ preserves the bound:
\begin{equation}\label{eq:6.lemmaHoef}
 \mathbb{P}\bigl(|B_{M_n}|\ge\varepsilon\bigr) \le 2\exp\left(-\frac{M_n\varepsilon^2}{32\Phi(0)^2}\right).
\end{equation}
Under the rate condition $M_n/\log n\to\infty$,
the series $\sum_{n\ge 1}\exp(-M_n\varepsilon^2/(32\Phi(0)^2))$ converges for every fixed $\varepsilon>0$,
because $M_n\varepsilon^2/(32\Phi(0)^2)\ge 2\log n$ for all sufficiently large $n$.
The first Borel--Cantelli lemma then yields
$\mathbb{P}(|B_{M_n}|\ge\varepsilon\;\text{infinitely often})=0$
for every $\varepsilon>0$,
hence $B_{M_n}\to 0$ as $n\to\infty$, a.s.

Combining the two parts, we arrive at $\hat{\gamma}_{M_n,N_n}^2\to \gamma_K(\mu,\nu)^2$, a.s.
For the stage-specific independent samples in the second part of the lemma, set
$D^{(n)}_{N_n}=(X^{(n)}_1,\ldots,X^{(n)}_{N_n},
Y^{(n)}_1,\ldots,Y^{(n)}_{N_n})$.  The bounded-kernel Hoeffding inequality gives
\[
 \mathbb P\!\left(\left|\bar\gamma_K^2(D^{(n)}_{N_n})-\gamma_K(\mu,\nu)^2\right|>\varepsilon\right)
 \le 2\exp(-c_KN_n\varepsilon^2)
\]
for a constant $c_K>0$.  The conditional feature bound
\eqref{eq:6.lemmaHoef} remains valid with the corresponding
$B^{(n)}_{M_n}$.  Both bounds are summable under
$N_n/\log n,M_n/\log n\to\infty$, so Borel--Cantelli gives convergence of
both error terms and completes the proof.
\qed

\subsection{Proof of Theorem~\ref{thm:convergence}}\label{sec:6.2}

Set
\[
 A_n:=\tfrac12\mathcal C(u_n)+c\mathcal R_T(u_n),\qquad
 G_n:=\lambda_n\gamma_K(\mathcal L(X_T^{u_n}),\mu_T)^2 .
\]
Both terms are non-negative and \eqref{eq:4.pop_opt} gives
\begin{equation}\label{eq:6.popupper}
 A_n+G_n\le H^*_c+\varepsilon_n .
\end{equation}
Hence $\{\mathcal C(u_n)\}$ is bounded and
$\gamma_K(\mathcal L(X_T^{u_n}),\mu_T)\to0$.

Let $Q^n:=\mathcal L(X^{u_n})$ on $C([0,T];\RR^d)$.  The energy bound and
\[
 \left|\int_s^t u_n(r,X_r)\,dr\right|
 \le\sqrt{t-s}\left(\int_0^T|u_n(r,X_r)|^2dr\right)^{1/2}
\]
combined with the fixed initial law and the usual uniform modulus estimate for
the Brownian term imply tightness of $\{Q^n\}$ by the modulus-of-continuity
criterion \cite[Theorem~7.3]{bil:1999}.  For a weakly convergent subsequence
$Q^n\to Q$, continuity of the evaluation maps, the preceding MMD
convergence, and Remark~\ref{rem:2.1} give $Q_T=\mu_T$.

Let $P^\sigma$ be the law of the uncontrolled diffusion.  Lemma~\ref{lem:entropy_energy}(i)
and lower semicontinuity of relative entropy yield
\[
 \sigma^2H(Q\,|\,P^\sigma)
 \le \liminf_{n\to\infty}\tfrac12\mathcal C(u_n).
\]
Lemma~\ref{lem:entropy_energy}(ii) produces an
admissible control system $\pi^Q\in\Pi_{\mu_T}$ having path law $Q$ and
energy $2\sigma^2H(Q\,|\,P^\sigma)$.  Since $W\in C_b$,
$\mathcal R[\mathcal L(X_t^{u_n})]\to\mathcal R[Q_t]$ for every $t$, and
bounded convergence gives
\[
 H^*_c\le J_c(\pi^Q)\le\liminf_{n\to\infty}A_n .
\]
The same conclusion holds for the full sequence by a sub-subsequence
argument, so $H_c^*\le\liminf_n A_n$.  On the other hand,
\eqref{eq:6.popupper} and $G_n\ge0$ give
$\limsup_nA_n\le H_c^*$.  Hence $A_n\to H_c^*$, and another application of
\eqref{eq:6.popupper} yields
$0\le G_n\le H_c^*+\varepsilon_n-A_n\to0$.
Thus $A_n+G_n\to H^*_c$, proving the theorem.
\qed

\medskip
\noindent\emph{Proof of Corollary~\ref{cor:empirical_evaluation}.}
Conditionally on $u_n$, the variables
$\int_0^T|u_n(t,X_t^{(n,i)})|^2dt$ are i.i.d., have mean
$\mathcal C(u_n)$, and take values in $[0,TC^2]$.  Hence Hoeffding's
inequality gives, for every $r>0$,
\[
 \mathbb P\!\left(
 \left|\hat{\mathcal C}_{N_n}(u_n)-\mathcal C(u_n)\right|>r
 \;\middle|\;u_n\right)
 \le 2\exp\!\left(-\frac{2N_nr^2}{T^2C^4}\right).
\]
The second condition in \eqref{eq:4.cond} and $\lambda_n\to\infty$ imply
$N_n/\log n\to\infty$.  Taking
$r_n=TC^2\sqrt{2\log n/N_n}$ makes the preceding upper bound $2n^{-4}$
and gives $r_n\to0$.  The Borel--Cantelli lemma therefore yields
\begin{equation}\label{eq:4.energy_eval}
 \hat{\mathcal C}_{N_n}(u_n)-\mathcal C(u_n)\longrightarrow0
 \qquad\text{a.s.}
\end{equation}
For the independent evaluation batch, the decomposition in the proof of
Lemma~\ref{lem:strong_consistency} and Hoeffding's inequality give, almost
surely for all sufficiently large $n$,
\begin{equation}\label{eq:6.eval_gamma}
 \left|\hat\gamma_{M_n,N_n}^2
 -\gamma_K(\mathcal L(X_T^{u_n}),\mu_T)^2\right|
 \le 12\Phi(0)\left(\sqrt{\frac{\log n}{M_n}}+
 \sqrt{\frac{\log n}{N_n}}\right).
\end{equation}
For the interaction, regard the controlled paths as random elements of
$C([0,T];\RR^d)$ and write exactly
\begin{equation}\label{eq:6.Rpath}
 \hat{\mathcal R}_T(u_n)
 =\frac1{N_n(N_n-1)}\sum_{i\ne j}h_{\tilde Z}(X^{(n,i)},X^{(n,j)}),
 \quad
 h_{\tilde Z}(x,y):=\frac{\Psi(0)}{2M_n}\sum_{r=1}^{M_n}
 \int_0^T\cos\bigl(\tilde Z_r^{\mathsf T}(x_t-y_t)\bigr)dt .
\end{equation}
Since $|h_{\tilde Z}|\le\Psi(0)T/2$, the same
decomposition used for the terminal estimator applies: conditional on the
evaluation paths, the random-feature average concentrates around the kernel
$U$-statistic, and the latter concentrates around $\mathcal R_T(u_n)$.
Both bounds are uniform in $u_n$.  A union bound followed by
Hoeffding--Borel--Cantelli therefore gives
\[
 |\hat{\mathcal R}_T(u_n)-\mathcal R_T(u_n)|
 \le3\Psi(0)T\left(\sqrt{\frac{\log n}{M_n}}+
 \sqrt{\frac{\log n}{N_n}}\right)\longrightarrow0 .
\]
Together with \eqref{eq:4.energy_eval} and \eqref{eq:4.cond},
these estimates imply
$\hat{\mathcal J}_{\lambda_n,c,M_n,N_n}(u_n)
-\mathcal J_{\lambda_n,c}(u_n)\to0$ a.s.; the conclusion follows from
Theorem~\ref{thm:convergence}.
\qed

\subsection{Proofs of Propositions~\ref{prop:rate}--\ref{prop:convergence_general}
and Corollary~\ref{cor:kernel_convergence}}\label{sec:6.3rate}

Dropping the non-negative energy and running-cost terms from
\eqref{eq:4.pop_opt} gives \eqref{eq:4.rate_mfg}, and its $c=0$, $T=1$
specialization gives the population bound in Proposition~\ref{prop:rate}.
Adding \eqref{eq:6.eval_gamma} gives \eqref{eq:4.rate_emp}; the polynomial
rate follows by substitution.  This proves Propositions~\ref{prop:rate}
and~\ref{prop:rate_mfg}.

For Proposition~\ref{prop:convergence_general}, weak continuity gives
$\mathcal R[Q_t^n]\to\mathcal R[Q_t]$ for each $t$.  Since $\mathcal R\ge0$,
Fatou's lemma yields
\[
 \int_0^T\mathcal R[Q_t]dt
 \le\liminf_{n\to\infty}\int_0^T\mathcal R[Q_t^n]dt,
\]
which replaces bounded convergence in the population proof.  For the sample-level
statement, Hoeffding's inequality for $U$-statistics yields, conditionally on $u_n$,
\[
 \mathbb P\!\left(
  |\hat{\mathcal R}_T(u_n)-\mathcal R_T(u_n)|>r
  \;\middle|\;u_n\right)
 \le 2\exp\!\left(-\frac{\lfloor N_n/m\rfloor r^2}{2B^2}\right),
 \qquad r>0.
\]
Since \eqref{eq:4.cond} implies $N_n/\log n\to\infty$, the
Hoeffding--Borel--Cantelli argument gives
\[
 \hat{\mathcal R}_T(u_n)-\mathcal R_T(u_n)\longrightarrow0
 \qquad\text{a.s.}
\]
This
replaces the interaction estimate following \eqref{eq:6.Rpath}; the energy
estimate follows from the uniform bound on $u_n$ exactly as in the proof of
Corollary~\ref{cor:empirical_evaluation}.

Finally, for the kernel estimator, Hoeffding's inequality for two-sample
$U$-statistics gives
\[
 \left|\bar\gamma_K^2(D_{N_n})
 -\gamma_K(\mathcal L(X_1^{u_n}),\mu_1)^2\right|
 \le8\Phi(0)\sqrt{\frac{2\log n}{N_n}}
\]
eventually almost surely.  The condition in
Corollary~\ref{cor:kernel_convergence} makes the right-hand side negligible
after multiplication by $\lambda_n$; the energy argument in the proof of
Corollary~\ref{cor:empirical_evaluation} and
Corollary~\ref{cor:sbp_special_case} complete the proof.
\qed
\section{Conclusion}\label{sec:7}

We have proposed a computational framework for the soft-target version of
\emph{mean-field optimal transport}, namely a potential / planning MFG
where both the running interaction cost and the terminal distributional cost are expressed
through reproducing-kernel maximum mean discrepancies, and applied it to a coordinated electric-vehicle charging problem.
The kernel estimators are sample-linear, with arithmetic cost $O(dNM)$ for one evaluation
($O(NM)$ when $d$ is fixed); the full pathwise iteration additionally includes the
$q$ time levels and the neural-network cost as detailed after Algorithm~\ref{alg:sbp}.
The feature-projection cost is therefore explicit and linear in $d$, which makes it compatible
with the high-dimensional Schr{\"o}dinger-bridge experiments at $d\le 100$
of Section~\ref{sec:5.3} while remaining moderate in the dimension of the MFG application reported here.
By combining the Fourier representation of the squared MMD and of the kernel self-interaction
with $U$-statistic decompositions,
we obtained unbiased estimators
$\hat\gamma_{M,N}^2$ (Theorem~\ref{thm:rfu_unbiased})
and $\hat{\mathcal{R}}_{M,N}[\hat\nu^N]$ (Theorem~\ref{thm:rfu_interaction})
of complexity $O(dNM)$ for the terminal MMD penalty and the interaction cost respectively.
Both estimators have an explicit variance decomposition $O(1/M)+O(1/N)$,
which directly underlies the convergence analysis of Section~\ref{sec:4.2conv}.

The theoretical content of the paper is summarized by four results:
\begin{enumerate}[label=(\arabic*)]
\item Two unbiased $O(dNM)$ estimators of the kernel costs of \eqref{eq:4.cost_decomp},
each with explicit variance decomposition $O(1/M)+O(1/N)$
(Theorems~\ref{thm:rfu_unbiased} and~\ref{thm:rfu_interaction}).
\item Convergence of population penalized near-minimizers to the constrained
potential-MFG value (Theorem~\ref{thm:convergence}), together with almost-sure
consistency of their objectives on independent evaluation batches
(Corollary~\ref{cor:empirical_evaluation}).
\item The explicit bounds
$\gamma_K^2\le(H^*_c+\varepsilon_n)/\lambda_n$ and
$\hat\gamma_{M_n,N_n}^2\le(H^*_c+\varepsilon_n)/\lambda_n+
O(\sqrt{\log n/M_n}+\sqrt{\log n/N_n})$,
yielding $\gamma_K=O(n^{-a/2})$ for $\lambda_n=n^a$
(Propositions~\ref{prop:rate} and~\ref{prop:rate_mfg}).
\item Recovery of the Schr{\"o}dinger-bridge problem of \cite{nak:2024sb} as the special case $\mathcal{R}\equiv 0$
(Corollary~\ref{cor:sbp_special_case}),
together with an independent sample-evaluation guarantee for the kernel-$U$-statistic implementation
(Corollary~\ref{cor:kernel_convergence}).
\end{enumerate}
Numerical experiments addressed two complementary settings:
the Schr{\"o}dinger-bridge specialization ($\mathcal{R}\equiv 0$) in dimensions up to $d=100$
with Gaussian shift and bimodal targets (Sections~\ref{sec:5.2}--\ref{sec:5.3})
and the supporting convergence diagnostics and computational-scaling experiments
(Sections~\ref{sec:conv_ver}, \ref{sec:5.5comp});
and the physically consistent EV charging extension of Section~\ref{sec:5ev},
where the same applied power drives the SOC and enters aggregate demand.
The EV results exhibit a Pareto trade-off between peak normalized charging power
and terminal service failures; they are not presented as an application of the
convergence theorem.

\paragraph{Unbiasedness in practice and in theory.}
The biased RFF $V$-statistic and the unbiased RF $U$-statistic terminal penalty produce empirically indistinguishable trained drifts across the configurations of Section~\ref{sec:5.1.3vs}.
We nevertheless adopt the unbiased construction for theoretical reasons:
it eliminates the $O(\Phi(0)/N)$ population-level penalty floor and yields the clean variance decomposition $O(1/M)+O(1/N)$
on which Proposition~\ref{prop:rate} relies.

\paragraph{Future directions.}
Several directions are worth pursuing.
First, the EV demonstration uses a single bounded physical heterogeneity coordinate;
enriching $h$ with additional vehicle-level parameters
(battery capacity, plug-out time uncertainty, comfort preference)
under $h$-dependent dynamics and a domain-justified distribution is immediate,
since the $O(N)$ and $O(dNM)$ scalings of our estimators carry over.
Second, extending the convergence theory to control-dependent price feedback,
degenerate frozen heterogeneity, and projected or reflected SOC dynamics would
place the EV model on the same theoretical footing as the state-law formulation.
Third, the adaptive bandwidth scaling $\alpha=1/d$ could be replaced by deep kernel learning
where the kernel operates in a learned feature space of lower effective dimension,
which would be useful for scaling to higher heterogeneity dimensions.
Fourth, other smart-grid problems (distributed battery storage, demand response, day-ahead market coordination)
and non-energy MFGs (crowd dynamics, opinion formation, multi-agent reinforcement learning)
share the structural form of \eqref{eq:1.mfg_intro}--\eqref{eq:1.kernel_costs}
and admit the same treatment.

\subsection*{Acknowledgements}

Supported by JSPS KAKENHI JP24K06861 and JP26H02001.

\subsection*{Use of AI assistants}
The author used generative-AI assistants (Anthropic Claude and OpenAI
Codex) for refactoring and commenting Python code, generating plotting
boilerplate, copy-editing, and assisting in the identification and
correction of proof gaps during revision.  The author independently
checked the resulting mathematical statements, proofs, citations,
code, and numerical results, and takes full responsibility for the
content of this paper.

\subsection*{Declarations}

\noindent\textbf{Funding.}
Supported by JSPS KAKENHI JP24K06861 and JP26H02001.

\noindent\textbf{Competing interests.}
The author has no competing interests to declare that are relevant to the content of this article.

\noindent\textbf{Ethics approval and consent to participate.}
Not applicable.

\noindent\textbf{Consent for publication.}
Not applicable.

\noindent\textbf{Code availability.}
The reference implementation reproducing all numerical experiments of Section~\ref{sec:5} is publicly available at \url{https://github.com/yumiharu-nakano/kernelMFG-RFU}.

\noindent\textbf{Data availability.}
All datasets used in the numerical experiments are synthetically generated by the reference implementation cited under Code availability; no external or third-party datasets were used.

\noindent\textbf{Author contributions.}
The sole author conceived the study, developed the theory, implemented the algorithms, performed the numerical experiments, and wrote the manuscript.

\bibliographystyle{spmpsci}
\bibliography{../mybib,mfg03_refs}

\end{document}